\documentclass[12pt]{amsart}
\usepackage{latexsym}
\usepackage{mathrsfs,amsmath,amssymb,amsthm,ascmac}
\usepackage[dvipdfmx]{graphicx,color}
\usepackage{wrapfig}
\usepackage[all]{xy}
\usepackage{tikz}
\usetikzlibrary{shapes.misc,shapes.geometric}

\theoremstyle{plain}
\newtheorem{theorem}{Theorem}[section]
\newtheorem{lemma}[theorem]{Lemma}
\newtheorem{prop}[theorem]{Proposition}
\newtheorem{fact}[theorem]{Fact}
\newtheorem{cor}[theorem]{Corollary}
\newtheorem{obs}[theorem]{Observation}
\newtheorem{claim}{Claim}

\newtheorem{question}{Question}

\theoremstyle{definition}
\newtheorem{definition}[theorem]{Definition}

\newtheorem{example}[theorem]{Example}

\newtheorem*{notation}{Notation}

\newcommand{\N}{\mathbb{N}}

\newcommand{\bfU}{\mathbf{U}}
\newcommand{\bfS}{\mathbf{S}}

\newcommand{\om}{\omega}
\newcommand{\Baire}{\om^\om}

\newcommand{\upto}{\upharpoonright}
\newcommand{\coupto}{\downharpoonright}

\newcommand{\fr}{{}^{\frown}}
\newcommand{\treeleq}{\trianglelefteq}

\newcommand{\mul}{\leadsto}
\newcommand{\ad}{\sqcup}
\newcommand{\syn}{{\tt Syn}}
\newcommand{\val}{{\tt val}}
\newcommand{\rank}{{\rm rank}}

\newcommand{\eval}[1]{[\hspace{-.15em}[{#1}]\hspace{-.15em}]}

\newcommand{\comU}{{U}}
\newcommand{\comr}{{u}}

\newcommand{\ep}{\varepsilon}

\newcommand{\sqleq}{\sqsubseteq}

\newcommand{\lang}{\mathcal{L}_{\rm Veb}}

\newcommand\tboldsymbol[1]{%
\protect\raisebox{0pt}[0pt][0pt]{%
$\underset{\widetilde{}}{\mathbf{#1}}$}\mbox{\hskip 1pt}}

\newcommand{\tpbf}[1]{\tboldsymbol{#1}}
\newcommand{\pcolon}{\colon\!\!\!\subseteq}

\newcommand{\bS}{\tpbf{\Sigma}^0}

\title[A syntactic approach to Borel functions]{A syntactic approach to Borel functions:\\Some extensions of Louveau's theorem}

\author{Takayuki Kihara and Kenta Sasaki}

\begin{document}

\begin{abstract}
Louveau showed that if a Borel set in a Polish space happens to be in a Borel Wadge class $\Gamma$, then its $\Gamma$-code can be obtained from its Borel code in a hyperarithmetical manner.
%
%Louveau showed that if two $\Sigma^1_1$ sets are separated by a set in a Borel Wadge class, then it is separated by a such a set which has a hyperarithmetical code.
We extend Louveau's theorem to Borel functions:
If a Borel function on a Polish space happens to be a $\tpbf{\Sigma}_t$-function, then one can effectively find its $\tpbf{\Sigma}_t$-code hyperarithmetically relative to its Borel code.
More generally, we prove extension-type, domination-type, and decomposition-type variants of Louveau's theorem for Borel functions.
\end{abstract}

\maketitle

\section{Introduction}

\subsection{Summary}
In \cite{Lou80}, Louveau showed that if a Borel set in a Polish space happens to be a $\tpbf{\Sigma}^0_\xi$ set, then its $\tpbf{\Sigma}^0_\xi$-code can be obtained from its Borel code in a hyperarithmetical manner.
This is derived from the so-called Louveau separation theorem \cite{Lou80}, which states that if a disjoint pair of $\Sigma^1_1$ sets is separated by a $\tpbf{\Sigma}^0_\xi$ set, then it is separated by a $\tpbf{\Sigma}^0_\xi$ set which has a hyperarithmetical $\tpbf{\Sigma}^0_\xi$-code.
Louveau applied his result to solve the section problem on Borel sets.
This result is useful for extracting information about uniformity from a non-uniform condition.
For instance, using Louveau's theorem, Solecki \cite{Sol98} obtained an inequality for cardinal invariants related to decomposability of Borel functions, and Fujita-M\'atrai \cite{FuMa10} solved Laczkovich's problem on differences of Borel functions.
Louveau's theorem is also known to be a powerful tool which enables us to use effective methods in topological arguments.
For instance, Gregoriades-Kihara-Ng \cite{GKN21} used it as a tool to reduce a problem on descriptive set theory to a problem on computability theory, and gave a partial solution to the decomposability problem on Borel functions.

In \cite{Lou83}, Louveau revisited his theorem in the context of the Wadge hierarchy.
The notion of Wadge degrees \cite{Wad83} provides us an ultimate refinement of all known hierarchies in descriptive set theory, such as the Borel hierarchy and the Hausdorff-Kuratowski difference hierarchy.
%Here, for topological spaces $X$ and $Y$, we say that a subset {\em $A\subseteq X$ is Wadge reducible to $B\subseteq Y$} (written $A\leq_WB$), if $A=f^{-1}[B]$ for some continuous function $f\colon X\to Y$.
%The definition looks quite elementary; however its surprisingly well-behaved structure was never revealed without the development of deep determinacy techniques \cite{Wad83}.
%If $X=Y=\om^\om$, the induced quotient poset, now called the Wadge degrees, turns out to provide us an ultimate refinement of all known hierarchies in descriptive set theory, such as the Borel hierarchy and the Hausdorff-Kuratowski difference hierarchy.
However, the original definition of the Wadge hierarchy does not tell us the way to obtain each Wadge class.
To address this issue, Louveau \cite{Lou83} introduced a set of basic $\om$-ary Boolean operations, and observed that each blueprint $u$ for how to combine these operations determines a Wadge class $\tpbf{\Gamma}_u$.
Indeed, Louveau \cite{Lou83} showed that any Borel Wadge class in the Baire space can be obtained as a combination of these $\om$-ary Boolean operations, and moreover, he showed that if a Borel set in a Polish space happens to be a $\tpbf{\Gamma}_u$ set, then its $\tpbf{\Gamma}_u$-code can be obtained from its Borel code in a hyperarithmetical manner.
Louveau's explicit description of each Borel Wadge class has been applied to prove Borel Wadge determinacy within second order arithmetic \cite{LoSR}.
Similar notions have also been investigated, e.g.~in \cite{Sel95,Dup01}.
For instance, Duparc \cite{Dup01,Dupxx} introduced a slightly different set of operations, on the basis of which he gave the normal form of each Borel Wadge degree, even in non-separable spaces.

The notion of Wadge reducibility has been extended to functions, and extensively studied, e.g.~in \cite{EMS,Sel07,Blo14,KM19,Her20}.
For Borel functions with better-quasi-ordered ranges, Kihara-Montalb\'an \cite{KM19} obtained a full characterization of the Wadge degrees, with a heavy use of the (suitably extended) nested labeled trees and forests.
In order to obtain the result, Kihara-Montalb\'an \cite{KM19} introduced a language consisting of a few basic algebraic operations, inspired by Duparc's operations \cite{Dup01,Dupxx}, for describing any Wadge classes of Borel functions.
As before, each blueprint $t$ for how to combine these operations determines a Wadge class $\tpbf{\Sigma}_t$.
The main difference with Louveau's work is not merely that it is for functions, but that the blueprints (i.e., the terms in the language) naturally form a quasi-ordered algebraic structure, which can be viewed as the structured collection of nested labeled trees and forests \cite{KM19}, and forgetting the algebraic information from the term model yields the ordering of Wadge degrees; see also \cite{Sel20,Kih}.

In this article, we show that if a Borel function between Polish spaces happens to be a $\tpbf{\Sigma}_t$ function, then its $\tpbf{\Sigma}_t$-code can be obtained from its Borel code in a hyperarithmetical manner.
To show this, we rewrite the definition of $\tpbf{\Sigma}_t$ as the mechanism controlled by certain flowcharts only involving conditional branching.
We observe that the flowchart representation of Borel functions is a powerful tool for clarifying various arguments on the class $\tpbf{\Sigma}_t$.
Indeed, this flowchart definition makes our proof simpler than Louveau's original one in \cite{Lou83}, even though our theorem is much more general than Louveau's one.
As another application of the flowchart representation, we also see that the topological and symbolic complexity of Borel functions on Polish spaces coincide:
The symbolic Wadge degrees \`{a} la Pequignot \cite{Peq15} on bqo-valued Borel functions on a (possibly higher-dimensional) Polish space is exactly the hierarchy obtained from the topological Wadge classes $\tpbf{\Sigma}_t$.

\subsection{Preliminaries}

In this article, we assume that the reader is familiar with elementary facts about descriptive set theory.
For the basics of (effective) descriptive set theory, we refer the reader to Moschovakis \cite{Mos09}.

For a function $f\colon X\to Y$ and $A\subseteq X$, we use the symbol $f\upto A$ denote the restriction of $f$ up to $A$.
We denote a partial function from $X$ to $Y$ as $f\pcolon X\to Y$.
We also use the following notations on strings:
For finite strings $\sigma,\tau\in\om^{<\om}$, we write $\sigma\preceq\tau$ if $\sigma$ is an initial segment of $\tau$, and write $\sigma\prec\tau$ if $\sigma$ is a proper initial segment of $\tau$.
We also use the same notation even if $\tau$ is an infinite string, i.e., $\tau\in\om^\om$.
For $\sigma\in\om^{<\om}$ and $\ell\in\om$, define $\sigma\upto \ell$ as the initial segment of $\sigma$ of length $\ell$.
For finite strings $\sigma,\tau\in\om^{<\om}$, let $\sigma\fr\tau$ be the concatenation of $\sigma$ and $\tau$.
If $\tau$ is a string of length $1$, i.e., $\tau$ is of the form $\langle n\rangle$ for some $n\in\om$, then $\sigma\fr\langle n\rangle$ is abbreviated to  $\sigma\fr n$.
We always assume that $\om^\om$ is equipped with the standard Baire topology, that is, the $\om$-product of the discrete topology on $\om$.
For $\sigma\in\om^{<\om}$, let $[\sigma]$ be the clopen set generated by $\sigma$, i.e., $[\sigma]=\{x\in\om^\om:\sigma\prec x\}$.

For Polish spaces $X$ and $Y$, for a pointclass $\Gamma$, we say that $f\colon X\to Y$ is {\em $\Gamma$-measurable} if $f^{-1}[U]\in\Gamma$ for any open set $U\subseteq Y$.
A {\em computable Polish space} or a {\em recursively presented Polish space} is a triple $(X,d,\alpha)$, where $(X,d)$ is a Polish space, $(\alpha_n)_{n\in\om}$ is a dense sequence in $X$, and the map $(n,m)\mapsto d(\alpha_n,\alpha_m)$ has a nice computability-theoretic property; see \cite[Section 3I]{Mos09} and \cite{CCA}.
If $\Gamma$ is a lightface pointclass on computable Polish spaces, then we say that $f\colon X\to Y$ is {\em effectively $\Gamma$-measurable}, or simply, {\em $\Gamma$-measurable} if the relation $R(x,e)$ defined by $f(x)\in B_e$ is in $\Gamma$, where $B_e$ is the $e$-th rational open ball in $Y$.
Let ${\sf WO}$ be the set of all well-orders on $\om$.
For $\alpha\in{\sf WO}$, we define $|\alpha|$ as the order type of $\alpha$.

Let $X$ and $Y$ be topological spaces, and $Q$ be a quasi-ordered set.
A function $f\colon X\to Q$ is {\em Wadge reducible to} $g\colon Y\to Q$ (written $f\leq_Wg$) if there exists a continuous function $\theta\colon X\to Y$ such that $f(x)\leq_Qg(\theta(x))$ for any $x\in X$.

%Here, for topological spaces $X$ and $Y$, we say that a subset {\em $A\subseteq X$ is Wadge reducible to $B\subseteq Y$} (written $A\leq_WB$), if $A=f^{-1}[B]$ for some continuous function $f\colon X\to Y$.
%The definition looks quite elementary; however its surprisingly well-behaved structure was never revealed without the development of deep determinacy techniques \cite{Wad83}.
%If $X=Y=\om^\om$, the induced quotient poset, now called the Wadge degrees, turns out to provide us an ultimate refinement of all known hierarchies in descriptive set theory, such as the Borel hierarchy and the Hausdorff-Kuratowski difference hierarchy.

\section{Describing Borel Wadge classes}

\subsection{Syntax}

\begin{figure}[t]
{\footnotesize
\includegraphics[width=3cm]{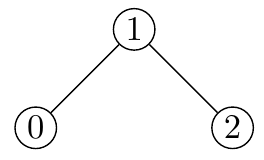}
\qquad\qquad\qquad\qquad
\includegraphics[width=4cm]{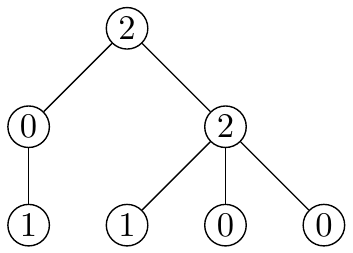}
\caption{\footnotesize (left) The tree $1\mul(0\ad 2)$;\quad (right) The tree $2\mul((0\mul 1)\ad(2\mul(1\ad 0\ad 0)))$}\label{figure:labeled-tree}
}
\end{figure}

As is well-known, the processes of approximating functions by finite mind-changes can be represented by labeled well-founded trees and forests; see e.g.~\cite{Her93,Sel07,Her20}.
As an algebraic aspect of such precesses, it is known that one can describe well-founded forests as terms in the signature $\{\ad,\mul\}$ over an (infinitary) equational theory (Figure \ref{figure:labeled-tree}; see also \cite{Sel20,Kih}).
A control mechanism for Borel functions is a bit more complicated, and it cannot be described by just a tree or a forest.
Instead, it is described by a matryoshka of trees: Within each node of the tree there is a tree, and within each node of that tree there is a tree, and within each node of that tree there is another tree, and so on.
Kihara-Montalb\'an \cite{KM19} introduced the language for describing matryoshkas of trees as follows:

\begin{definition}[Kihara-Montalb\'an \cite{KM19}]
For a set $Q$, the signature $\lang(Q)$ consists of a constant symbol $q$ for each $q\in Q$, a binary function symbol $\mul$, an $\om$-ary function symbol $\ad$, and a unary function symbol $\langle\cdot\rangle^{\om^\alpha}$ for each $\alpha<\om_1$.
If $Q=\emptyset$, we write $\lang$ for $\lang(\emptyset)$.
\end{definition}

Here, the original language introduced by Kihara-Montalb\'an \cite{KM19} uses the symbol ${}^\to$ instead of $\mul$ (the reason for using the symbol $\mul$ here is to avoid confusion, since the usual arrow symbol is used for many other purposes).
As in \cite{KM19}, we abbreviate the symbol $\langle\cdot\rangle^{\om^0}$ as $\langle\cdot\rangle$.
The description $\langle t\rangle$ represents a node labeled by $t$.
Not only that, but the label types are ranked, e.g.~$\langle t\rangle^{\om^\alpha}$ is a node labelled with $t$, and the rank of its label is $\om^\alpha$.
Hereafter, we use the symbol $\phi_\alpha$ to denote $\langle\cdot\rangle^{\om^\alpha}$.
As the function symbol $\phi_\alpha=\langle\cdot\rangle^{\om^\alpha}$ is specified by the fixed point axiom $\phi_\beta(\phi_\alpha(t))=\phi_\alpha(t)$ for any $\beta<\alpha$, we say that $\phi_\alpha$ is the {\em $\alpha$-th Veblen function symbol}, see also \cite{KS19,Kih}, where the symbol $s_\alpha$ is used in \cite{KS19,Sel20} instead of $\langle\cdot\rangle^{\om^\alpha}$ or $\phi_\alpha$.

The notion of terms in a given signature is inductively defined as usual in logic, universal algebra, and other areas.
In this article, the analysis of the syntax of terms plays an important role in various aspects.

\begin{definition}[Syntax tree]
Let $\mathcal{L}$ be a (possibly infinitary) signature.
Any $\mathcal{L}$-term $t$ defines a labeled tree $\syn_t=({\tt S}_t,{\tt l}_t)$ called the {\em syntax tree} of the term $t$ as follows:
\begin{itemize}
\item If $t$ is a constant or variable symbol $x$, then the syntax tree $\syn_t$ is the singleton ${\tt S}_t=\{\ep\}$ labeled by ${\tt l}_t(\ep)=x$, i.e., $\syn_t$ can be written as $\langle x\rangle$ in terms of the language of labeled trees.
%$\lang(\mathcal{L}\cup{\rm Var})$.
\item If $t$ is of the form $f(\langle s_i\rangle_{i\in I})$ for some function symbol $f\in\mathcal{L}$ and $\mathcal{L}$-terms $\langle s_i\rangle_{i\in I}$, then $\syn_t$ is the result by adding a root $\ep$ and edges from $\ep$ to $\syn_{s_i}$ for $i\in I$, where the root $\ep$ is labeled by $f$, i.e., $\syn_t$ can be written as $\langle f\rangle\mul\ad_{i\in I}\syn_{s_i}$ in terms of the language of labeled trees.
%$\lang(\mathcal{L}\cup{\rm Var})$.
\end{itemize}
\end{definition}

\begin{figure}[t]
{\footnotesize
\includegraphics[width=4cm]{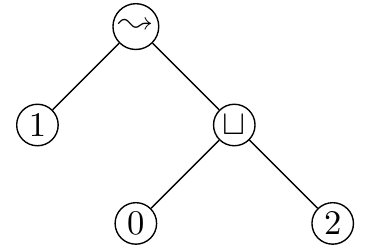}
\caption{\footnotesize The syntax tree of $1\mul(0\ad 2)$}\label{figure:syntax-tree}
}
\end{figure}

Note that the syntax tree of an $\mathcal{L}$-term is always well-founded.
When considering $\lang(Q)$-terms, be careful not to confuse the syntax tree of a term $t$ (e.g., Figure \ref{figure:syntax-tree}) with the tree represented by $t$ (e.g., Figure \ref{figure:labeled-tree}).
If $t$ is an $\lang(Q)$-term, then one can always think of $\syn_t$ as a labeled subtree of $\om^{<\om}$, since the signature $\lang(Q)$ consists of at most $\om$-ary function symbols.
More explicitly, for any $\lang(Q)$-term $t$, one may assume that ${\tt S}_t$ is a subtree of $\om^{<\om}$ as follows:
\begin{itemize}
\item If $t$ is of the form $q$ then ${\tt S}_t=\{\langle\rangle\}$, where $\langle\rangle$ is the empty string.
\item If $t$ is of the form $u\mul s$ then ${\tt S}_t=\{0\fr\sigma:\sigma\in{\tt S}_u\}\cup\{1\fr\sigma:\sigma\in{\tt S}_s\}$.
\item If $t$ is of the form $\ad_{n\in\om}s_n$ then ${\tt S}_t=\{n\fr\sigma:n\in\om\mbox{ and }\sigma\in{\tt S}_{s_n}\}$.
\item If $t$ is of the form $\phi_\alpha(s)$ then ${\tt S}_t=\{0\fr\sigma:\sigma\in{\tt S}_{s}\}$.
\end{itemize}

Recall that a term is {\em closed} if it does not contain variable symbols.
As in Kihara-Montalb\'an \cite{KM19}, we consider only closed $\lang(Q)$-terms in most cases, so we refer to a closed $\lang(Q)$-term simply as an $\lang(Q)$-term.
When dealing with an $\lang(Q)$-term which may contain variable symbols, we refer to it emphatically as an {\em open} $\lang(Q)$-term.

\begin{definition}
An $\lang(Q)$-term $t$ is {\em well-formed} if $t$ contains no subterm of the form $\phi_\alpha(\ad_{n\in\om}s_n)$; and $t$ is {\em normal} if the symbol $\mul$ always occurs as of the form $\phi_\alpha(u)\mul \ad_{n\in\om}s_n$ or ${q}\mul \ad_{n\in\om}s_n$.
\end{definition}

In other words, using a syntax tree, an $\lang(Q)$-term $t$ is well-formed if and only if, for any $\sigma\in\syn_t$, the following holds:
\[\mbox{$\sigma$ is labeled by $\phi_\alpha$}\implies\mbox{$\sigma\fr 0$ is not labeled by $\sqcup$}.\]

Similarly, an $\lang(Q)$-term $t$ is normal if and only if, for any $\sigma\in\syn_t$, the following holds:
\[\mbox{$\sigma$ is labeled by $\mul$}\implies
\begin{cases}
\mbox{$\sigma\fr 0$ is either a leaf or labeled by $\phi_\alpha$},\\
\mbox{$\sigma\fr 1$ is labeled by $\ad$}.
\end{cases}\]

%In other words, $t$ is well-formed if and only if $t$ contains no subterm of the form $\phi_\alpha(\ad_{n\in\om}s_n)$; and $t$ is normal if and only if the symbol $\mul$ always occurs as of the form $\phi_\alpha(u)\mul \ad_{n\in\om}s_n$ or ${q}\mul \ad_{n\in\om}s_n$.
In Kihara-Montalb\'an \cite{KM19}, only normal well-formed $\lang(Q)$-terms have been considered; however, a certain (infinitary) equational theory proves that every $\lang(Q)$-term is equal to a normal $\lang(Q)$-term; see \cite{Kih,Sel19b}.

Let $\treeleq$ be the nested homomorphic quasi-order on the normal well-formed $\lang(Q)$-terms introduced by Kihara-Montalb\'an \cite{KM19} to characterize the order type of the Wadge degrees of $Q$-valued Borel functions.

\begin{fact}[Kihara-Montalb\'an {\cite[Theorem 1.5]{KM19}}]\label{fact:KM-main}
Let $Q$ be a better-quasi-ordered set.
Then, the Wadge degrees of Borel functions $\om^\om\to Q$ is isomorphic to the quotient of the quasi-order $\treeleq$ on the well-formed $\lang(Q)$-terms.
\end{fact}

\subsubsection*{Coding a syntax tree:}

As usual, a subset of $\om^{<\om}$ can be viewed as a subset of $\om$ via an effective bijection $\om^{<\om}\simeq\om$.
We assume that the constant symbols $Q$ and the variable symbols are indexed by $\om^\om$, i.e., $Q=\{q_z\}_{z\in\om^\om}$ and ${\rm Var}=\{x_z\}_{z\in\om^\om}$, and then we consider the following coding:
\[{\tt c}(0,z)=\langle q_z\rangle,\quad{\tt c}(1,z)=x_z,\quad{\tt c}(2,z)=\ \mul,\quad{\tt c}(3,z)=\ad,\quad{\tt c}(4,z)=\phi_{|z|},\]
where ${\tt c}(4,z)$ is defined only when $z\in{\sf WO}$.
The function ${\tt c}$ gives a representation of the $\lang(Q)$-symbols.
A realizer of a labeling function ${\tt l}_t$ is a function $\lambda_t\colon{\tt S}_t\to\om\times\om^\om$ such that ${\tt l}_t={\tt c}\circ\lambda_t$.
As we consider ${\tt S}_t$ as a subset of $\om$, a realizer $\lambda_t$ of a labeling function can be viewed as an element of $(\om\times\om^\om)^\om\simeq\om^\om$.
Via this identification, we call such a pair $({\tt S}_t,\lambda_t)$ (as an element of $\om^\om$) a code of the syntax tree $\syn_t=({\tt S}_t,{\tt l}_t)$.

As for effectivity, we say that an $\lang(Q)$-term $t$ is hyperarithmetic or $\Delta^1_1$ if its syntax tree $\syn_{t}$ has a $\Delta^1_1$-code.

\subsection{Semantics}

The essential idea of Kihara-Montalb\'an's Theorem (Fact \ref{fact:KM-main}) is to map the $\lang(Q)$-terms to the control mechanisms of Borel functions.
More precisely, Kihara-Montalb\'an assigned to each $\lang(Q)$-term $t$ a class $\tpbf{\Sigma}_t$ of Borel functions, which may be thought of as an ultimate refinement of the Borel/Baire hierarchy.

\begin{definition}[First definition of $\tpbf{\Sigma}_t$]\label{def:K-M-Sigma}
Let $Z$ be a zero-dimensional Polish space.
For each $\lang(Q)$-term $t$, we inductively define the classes $\tpbf{\Sigma}_t(Z)$ and $\tpbf{\Sigma}_t(Y;Z)$ of $Q$-valued functions on Polish spaces $Z$ and $Y\subseteq Z$ as follows:
\begin{enumerate}
\item $\tpbf{\Sigma}_{{q}}(Z)$ consists only of the constant function $x\mapsto q\colon Z\to Q$.
%\item If $t=\ad_is_i$, then $f\in\tpbf{\Sigma}_t(X)$ if and only if there are an open cover $(U_i)_{i\in\N}$ of $X$ and a sequence $(\gamma_i)_{i\in\om}$ of continuous functions $\gamma_i\colon U_i\to X$ such that $f\upto U_i=g\circ\gamma_i$ for some $g_i\in\tpbf{\Sigma}_s(X)$.
%\item If $t=r\mul s$, then $f\in\tpbf{\Sigma}_t(X)$ if and only if there are an open set $V\subseteq X$, continuous functions $\gamma_0\colon (X\setminus V)\to X$ and $\gamma_1\colon V\to X$ such that $f\upto V=g_1\circ\gamma_1$ and $f\upto(X\setminus V)=g_0\circ\gamma_0$ for some $g_0\in\tpbf{\Sigma}_r(X)$ and $g_1\in\tpbf{\Sigma}_s(X)$.
\item If $t=\ad_{i\in\om}s_i$, then $f\in\tpbf{\Sigma}_t(Z)$ if and only if there is an open cover $(U_i)_{i\in\om}$ of $Z$ such that $f\upto U_i\in\tpbf{\Sigma}_{s_i}(U_i;Z)$ for each $i\in\N$.
\item If $t=s\mul u$, then $f\in\tpbf{\Sigma}_t(Z)$ if and only if there is an open set $V\subseteq Z$ such that $f\upto V\in\tpbf{\Sigma}_u(V;Z)$ and $f\upto(Z\setminus V)=g\upto(Z\setminus V)$ for some $g\in\tpbf{\Sigma}_s(Z)$.
\item If $t=\phi_\alpha(s)$, then $f\in\tpbf{\Sigma}_t(Z)$ if and only if there is a $\tpbf{\Sigma}^0_{1+\om^\alpha}$-measurable function $\beta\colon Z\to Z$ and a $\tpbf{\Sigma}_s(Z)$ function $g\colon Z\to Q$ such that $f=g\circ\beta$.
\end{enumerate}

Here, a function $f\colon Y\to Q$ is in $\tpbf{\Sigma}_t(Y;X)$ if there exists a continuous function $\gamma\colon Y\to X$ such that $f=g\circ\gamma$ for some $g\in\tpbf{\Sigma}_t(X)$.
\end{definition}

In this article, we introduce various equivalent definitions of $\tpbf{\Sigma}_t$.
When we emphasize that it is $\tpbf{\Sigma}_t$ in the sense of Definition \ref{def:K-M-Sigma}, we write it as $\tpbf{\Sigma}_t^W$;
the superscript $W$ suggests that this is the one associated with a Wadge class.
Indeed, Kihara-Montalb\'an's theorem in \cite{KM19} shows that any Borel Wadge class can be described as $\tpbf{\Sigma}_t^W$ for some $\lang(Q)$-term $t$.
That is to say, the essential ingredient of Fact \ref{fact:KM-main} is not simply that the two structures are isomorphic, but more precisely that the complexity of the control mechanism of the Borel functions represented by a term corresponds exactly to the Wadge degree.
This can be stated as follows:

\begin{fact}[Kihara-Montalb\'an {\cite[Propositions 1.7 and 1.8]{KM19}}]\label{fact:KM-main2}
Let $Q$ be a better-quasi-ordered set, and $t$ be an $\lang(Q)$-term.
If $g\colon \Baire\to Q$ is a $\tpbf{\Sigma}_t^W$-function, then for any $f\colon\Baire\to Q$,
\[f\leq_W g\iff f\in\tpbf{\Sigma}^W_s(\Baire)\mbox{ for some $s\treeleq t$}.\]
\end{fact}

Note that Definition \ref{def:K-M-Sigma} is slightly different from Kihara-Montalb\'an's original definition \cite{KM19} of the class $\tpbf{\Sigma}_t$.
To explain the original definition of the class $\tpbf{\Sigma}_t$, we assume $Z=\om^\om$, and let us consider the following:
First, given an open set $V\subseteq Z$ one can effectively construct a map $e_V\colon\om\to\om^{<\om}$ such that $V=\bigcup_{n\in\om}[e_V(n)]$, and $e_V(n)$ is incomparable with $e_V(m)$ whenever $n\not=m$.
Then, define ${\tt in}_V(n\fr x)=e_V(n)\fr x$.
It is easy to see that ${\tt in}_V\colon\om^\om\simeq V$ is a homeomorphism.
Next, by zero-dimensionality of $Z$, one can also effectively find a continuous retraction ${\tt out}_V\colon\om^\om\to(\om^\om\setminus V)$.
See also Kihara-Montalb\'an \cite[Observations 3.5 and 3.6]{KM19} for the details.

\begin{definition}[Second definition of $\tpbf{\Sigma}_t$ {\cite[Definition 3.7]{KM19}}]\label{def:K-M-Sigma2}
For each $\lang(Q)$-term $t$, we inductively define $\tpbf{\Sigma}_t(Z)$ as the class fulfilling (1), (2$^\prime$), (3$^\prime$) and (4):
\begin{enumerate}
\item[(2$^\prime$)]
If $t=\ad_is_i$, then $f\in\tpbf{\Sigma}_t(Z)$ if and only if there is a clopen partition $(C_i)_{i\in\N}$ of $Z$ such that $f\circ {\tt in}_{C_i}\in\tpbf{\Sigma}_{s_i}(Z)$ for each $i\in\N$.
\item[(3$^\prime$)]
If $t=s\mul u$, then $f\in\tpbf{\Sigma}_t(Z)$ if and only if there is an open set $V\subseteq Z$ such that $f\circ {\tt in}_V\in\tpbf{\Sigma}_u(Z)$ and $f\circ{\tt out}_V\in\tpbf{\Sigma}_s(Z)$.
\end{enumerate}
\end{definition}

To avoid confusion, we write $\tpbf{\Sigma}_t^{W\circ}$ for $\tpbf{\Sigma}_t$ in the sense of Definition \ref{def:K-M-Sigma2}.

\begin{obs}
If $Z=\om^\om$, the first and second definitions of $\tpbf{\Sigma}_t(Z)$ coincide, i.e., $\tpbf{\Sigma}^W_t(Z)=\tpbf{\Sigma}_t^{W\circ}(Z)$.
\end{obs}

\begin{proof}
We first show the inclusion $\tpbf{\Sigma}_t^W(Z)\subseteq\tpbf{\Sigma}_t^{W\circ}(Z)$ by induction.
If $t=\ad_{i\in\om}s_i$, and $f\in\tpbf{\Sigma}^W_t(Z)$ via an open cover $(U_i)$ of $Z$, then $f\upto U_i=g_i\circ\gamma_i$ for some $g_i\in\tpbf{\Sigma}^W_{s_i}(Z)$ and $\gamma_i$.
By the induction hypothesis, we have $g_i\in\tpbf{\Sigma}^{W\circ}_{s_i}(Z)$.
Since $Z$ is zero-dimensional, one can effectively find a clopen partition $(C_i)$ such that $C_i\subseteq U_i$ and $\bigcup_iC_i=\bigcup_iU_i$.
Then, we get $f\circ{\tt in}_{C_i}=g_i\circ\gamma_i\circ{\tt in}_{C_i}\in\tpbf{\Sigma}_{s_i}^{W\circ}(Z)$ since $\gamma_i\circ{\tt in}_{C_i}\colon Z\to Z$ is continuous.
Hence, $f\in\tpbf{\Sigma}_t^{W\circ}(Z)$.
If $t=s\mul u$, and $f\in\tpbf{\Sigma}_t^W(Z)$ via an open set $V\subseteq Z$, then $f\upto (Z\setminus V)=g_0\circ\gamma_0$ for some $g_0\in\tpbf{\Sigma}_s^W(Z)$ and $\gamma_0$, and $f\upto V=g_1\circ\gamma_1$ for some $g_1\in\tpbf{\Sigma}_u^W(Z)$ and $\gamma_1$.
By the induction hypothesis, we have $g_0\in\tpbf{\Sigma}_s^{W\circ}(Z)$ and $g_1\in\tpbf{\Sigma}_u^{W\circ}(Z)$.
Then, we get $f\circ{\tt out}_V=g_0\circ\gamma_0\circ{\tt out}_V\in\tpbf{\Sigma}_s^W(Z)$ since $\gamma_0\circ{\tt out}_V\colon Z\to Z$ is continuous, and similarly, we also get $f\circ{\tt in}_V=g_1\circ\gamma_1\circ{\tt in}_V\in\tpbf{\Sigma}_u^W(Z)$ since $\gamma_1\circ{\tt in}_V\colon Z\to Z$ is continuous.

Next, we show the inclusion $\tpbf{\Sigma}_t^{W\circ}(Z)\subseteq\tpbf{\Sigma}_t^W(Z)$ by induction.
If $t=\ad_{i\in\om}s_i$, and if $f\in\tpbf{\Sigma}_t^{W\circ}(Z)$ via a clopen partition $(C_i)_{i\in\om}$ of $Z$, then by the induction hypothesis, one can see that $f\upto C_i\in\tpbf{\Sigma}_{s_i}^W(C_i;Z)$ via $g_i=f\circ{\tt in}_{C_i}$ and $\gamma_i={\tt in}_{C_i}^{-1}$.
Hence, $f\in\tpbf{\Sigma}_t^W(Z)$.
If $t=s\mul u$, and $f\in\tpbf{\Sigma}_t^{W\circ}(Z)$ via an open set $V\subseteq Z$, then by the induction hypothesis, one can see that $f\upto V\in\tpbf{\Sigma}_u^W(V;Z)$ via $g_0=f\circ{\tt out}_V$ and $\gamma_0={\rm id}$, and $f\upto(Z\setminus V)\in\tpbf{\Sigma}_s^W(Z\setminus V;Z)$ via $g_1=f\circ{\tt in}_V$ and $\gamma_1={\tt in}_V^{-1}$.
Hence, $f\in\tpbf{\Sigma}_t^W(Z)$.
\end{proof}

\subsection{Command}

Next, we decompose the inductive definition of $\tpbf{\Sigma}_t$ along the syntax tree of $t$.
This decomposition reveals the algorithmic aspect of the definition of $\tpbf{\Sigma}_t$.
Let $t$ be an $\lang(Q)$-term, and consider its syntax tree $\syn_t$.

\begin{definition}
A {\em command} on the $\lang(Q)$-term $t$ over a zero-dimensional Polish space $Z$ is a family $\bfU=(\comU_\sigma,\comr_\sigma)_{\sigma\in\syn_t}$ indexed by the syntax tree $\syn_t$ satisfying the following condition:
%\begin{enumerate}
%\item Put $X_{\langle\rangle}=X$.
%\item For a leaf $\sigma\in\syn_t$, $\theta_\sigma$ is not specified.
%\item If $\sigma\in\syn_t$ is labeled by $\mul$, then $\theta_\sigma$ is an open subset $U$ of $X_{\sigma}$, and $X_{\sigma'}=X_{\sigma}$ for each  immediate successor $\sigma'$ of $\sigma$.
%\item If $\sigma\in\syn_t$ is labeled by $\ad$, then $\theta_\sigma$ is a countable sequence $(U_n)$ of open subsets of $X_{\sigma}$, and $X_{\sigma'}=X_{\sigma}$ for each  immediate successor $\sigma'$ of $\sigma$.
%\item If $\sigma\in\syn_t$ is labeled by $\phi_\alpha$, then $\theta_\sigma$ is a partial $\tpbf{\Sigma}^0_{1+\omega^\alpha}$-measurable function $u\pcolon X_{\sigma}\to X_{\sigma'}$, where $\sigma'$ is the unique immediate successor of $\sigma$ and $X_{\sigma'}$ is a Polish space.
%\end{enumerate}
\begin{enumerate}
\item For a leaf $\sigma\in\syn_t$, then $\comU_\sigma=Z$.
\item If $\sigma\in\syn_t$ is labeled by $\mul$, then $\comU_\sigma$ is an open subset of $Z$, $\comr_{\sigma\fr 0}$ is the identity map, and $\comr_{\sigma\fr 1}\colon \comU_\sigma\to Z$ is a continuous function.
\item If $\sigma\in\syn_t$ is labeled by $\ad$, then $\comU_\sigma$ is a sequence $(\comU_{\sigma,n})_{n\in\om}$ of open subsets of $Z$, and $\comr_{\sigma\fr n}\colon \comU_{\sigma,n}\to Z$ is a continuous function for each $n\in\om$.
\item If $\sigma\in\syn_t$ is labeled by $\phi_\alpha$, then $\comU_\sigma=Z$, and $\comr_{\sigma\fr 0}\colon Z\to Z$ is a $\tpbf{\Sigma}^0_{1+\omega^\alpha}$-measurable function.
\end{enumerate}

Note that Definition \ref{def:K-M-Sigma} (2) requires an additional condition:
We say that a command $\bfU$ is {\em strongly total} if $(\comU_{\sigma,n})_{n\in\om}$ is an open cover of $Z$ whenever $\sigma$ is labeled by $\ad$.
\end{definition}

A command may be thought of as a control flow that defines a function:
Feed a value into an input variable ${\tt x}$.
%Then follow the command from the root to a leaf, where each $(\comU_\sigma,\comr_\sigma)$ represents the condition branch instruction ``if ${\tt x}\in \comU_\sigma$ then '' in the case (2) 
Then follow the command from the root to a leaf, where each $\comU_\sigma$ represents a condition branch depending on ${\tt x}\in \comU_\sigma$ in the case (2) and on ${\tt x}\in U_{\sigma,n}$ in the case (3), and $\comr_{\sigma\fr n}$ represents a reassignment ${\tt x}\leftarrow \comr_{\sigma\fr n}({\tt x})$.
When we reach a leaf, output the label of the leaf, where note that any leaf must be labeled by a constant symbol $q\in Q$.
%However, the function $[\theta]$ induced by a command $\theta$ may be partial and multi-valued.

To give a formal definition of the above argument, for a node $\sigma\in\syn_t$ of length $\ell>0$, put
\[\val_\sigma=\comr_{\sigma}\circ\comr_{\sigma\upto{\ell-1}}\circ\dots\circ\comr_{\sigma\upto 2}\circ\comr_{\sigma\upto 1}\pcolon Z\to Z,\]
and $\val_\ep={\rm id}$, where $\ep$ is the empty string.
In other words, if $x$ is the first value fed into variable ${\tt x}$, then $\val_\sigma(x)$ is the value stored in variable ${\tt x}$ when we reach $\sigma$.
Next, we introduce the notion of a {\em true position} for $x\in Z$ (with respect to the command $\bfU$) in the following inductive manner, where we consider $\syn_t$ as a labeled subtree of $\om^{<\om}$:
\begin{itemize}
\item The root of the syntax tree $\syn_t$ is a true position for $x$.
\item Assume that $\sigma\in\syn_t$ is a true position for $x$, and is labeled by $\mul$.
If $\val_\sigma(x)\in\comU_\sigma$ then $\sigma\fr 1$ is a true position for $x$, and if $\val_\sigma(x)\not\in\comU_\sigma$ then $\sigma\fr 0$ is a true position for $x$.
\item Assume that $\sigma\in\syn_t$ is a true position for $x$, and is labeled by $\ad$.
If $\val_\sigma(x)\in U_{\sigma,n}$ then $\sigma\fr n$ is a true position for $x$.
\item Assume that $\sigma\in\syn_t$ is a true position for $x$, and is labeled by $\phi_\alpha$.
Then, the unique immediate successor $\sigma\fr 0$ of $\sigma$ is a true position for $x$.
%このとき，$i_\sigma(x)=\theta_\sigma(i_{\sigma^-}(x))$とする．
\end{itemize}

If a leaf $\rho\in\syn_t$ is a true position for $x$, then we call it a {\em true path for $x$} (with respect to $\bfU$).
A label of a true path for $x$ is considered as an output of the function determined by the command $\bfU$.

\begin{notation}
We write $\eval{\bfU}(x)=q$ if $q$ is a label of a true path for $x$ with respect to $\bfU$.
\end{notation}

However, it may happen that there are many true paths for $x$ or there is no true path for $x$.
If, for any $x\in Z$, there is a true path for $x$, that is, $\eval{\bfU}(x)$ is defined, then we call $\bfU$ {\em total}.
It is easy to see that a strongly total command is total.
If, for any $x\in Z$, there are at most one value $q$ such that $\eval{\bfU}(x)=q$, then $\bfU$ is called a {\em deterministic command}.
Then, the function $\eval{\bfU}\colon Z\to Q$ is called the {\em evaluation} of the command $\bfU$.
%A single-valued command is {\em total} on $X$ if $[\theta](x)$ is defined for any $x\in X$.

\begin{obs}
Let $Z$ be a zero-dimensional Polish space.
For any $\lang(Q)$-term $t$, the class $\tpbf{\Sigma}_t^W(Z)$ is the set of functions of the evaluations $\eval{\bfU}$ of deterministic strongly total commands $\bfU$ on the term $t$ over the space $Z$.
\end{obs}

One may consider $\eval{\bfU}$ even if $\bfU$ is non-deterministic, and in such a case, $\eval{\bfU}$ is multi-valued.
Next, we say that a command $\theta$ is {\em simple} if $u_{\sigma\fr i}$ is the identity map for each $i$ whenever $\sigma$ is labeled by either $\ad$ or $\mul$.
Later we will see that the deterministic total simple commands also yield the same class.
This notion induces the third inductive definition of $\tpbf{\Sigma}_t$:

\begin{definition}[Third definition of $\tpbf{\Sigma}_t$]\label{def:K-M-Sigma3}
For each $\lang(Q)$-term $t$, we inductively define $\tpbf{\Sigma}_t(Z)$ as the class fulfilling (1), (2$^{\prime\prime}$), (3$^{\prime\prime}$) and (4):
\begin{enumerate}
\item[(2$^{\prime\prime}$)]
If $t=\ad_is_i$, then $f\in\tpbf{\Sigma}_t(Z)$ if and only if there is an open cover $(U_i)_{i\in\N}$ of $Z$ such that $f\upto U_i\in\tpbf{\Sigma}_{s_i}(U_i)$ for each $i\in\N$.
\item[(3$^{\prime\prime}$)]
If $t=s\mul u$, then $f\in\tpbf{\Sigma}_t(Z)$ if and only if there is an open set $V\subseteq X$ such that $f\upto V\in\tpbf{\Sigma}_u(V)$ and $f\upto (Z\setminus V)\in\tpbf{\Sigma}_s(Z\setminus V)$.
\end{enumerate}
\end{definition}

To avoid confusion, let us write $\tpbf{\Sigma}_t'$ for $\tpbf{\Sigma}_t$ in the sense of Definition \ref{def:K-M-Sigma3}.
This third definition seems to be the most natural definition of $\tpbf{\Sigma}_t$ among the ones given so far.
It should be noted, however, that this third definition has a very different feature from the first two:
In the process of defining a function, a partial function with a complicated domain may appear.
Because of this feature, proving the equivalence of first and second definitions and the third definition is not an easy task (see Theorem \ref{prop:matryo-flow}), and for this reason, the relevance of the third definition to the Wadge theory is not immediately obvious.

\subsubsection*{Borel rank:}
We define the {\em Borel rank} of a node $\sigma\in\syn_t$ as follows.
First, enumerate all proper initial segments of $\sigma$ which are labeled by Veblen function symbols $\phi_\alpha$:
\[\tau_0\prec\tau_1\prec\tau_2\prec\dots\prec\tau_\ell\prec\sigma,\]
where $\tau_i$ is labeled by $\phi_{\alpha_i}$.
We call $(\tau_i)_{i\leq\ell}$ the {\em Veblen initial segments of $\sigma$}.
Then, the Borel rank of $\sigma\in\syn_t$ is defined as the following ordinal:
\[{\rm rank}(\sigma)=1+\om^{\alpha_0}+\om^{\alpha_{1}}+\dots+\om^{\alpha_{\ell}}.\]

\begin{obs}\label{obs:val-measurable}
The function $\val_\sigma\pcolon Z\to Z$ is $\tpbf{\Sigma}^0_{{\rm rank}(\sigma)}$-measurable.
Moreover, the domain of $\val_\sigma\pcolon Z\to Z$ is $\tpbf{\Sigma}^0_{{\rm rank}(\sigma)}$.
\end{obs}

\begin{proof}
The first assertion is clear.
It is easy to show the second assertion by induction since the domain of $\comr_\sigma$ is open for each $\sigma\in\syn_t$.
\end{proof}

\subsubsection*{Coding a command:}

Using the standard codings of the open sets and the $\tpbf{\Sigma}^0_{1+\om^\alpha}$-measurable functions, one can define the notion of a code of a command $\bfU$ on an $\lang(Q)$-term $t$ in the straightforward manner.
More precisely, a {\em code} of a command $\bfU=(\comU_\sigma,\comr_\sigma)_{\sigma\in\syn_t}$ is a pair of a code of the syntax tree $\syn_t=({\tt S}_t,{\tt l}_t)$ and an ${\tt S}_t$-indexed collection $(c_\sigma)_{\sigma\in{\tt S}_t}$ of elements of $\om^\om$ satisfying the following conditions:
\begin{enumerate}
\item If $\sigma$ is a leaf, then $c_\sigma$ is arbitrary.
\item If $\sigma$ is labeled by $\mul$, then $c_\sigma$ is a triple of $\tpbf{\Sigma}^0_1$-codes of $(\comU_\sigma,\comr_{\sigma,0},\comr_{\sigma,1})$.
\item If $\sigma$ is labeled by $\ad$, then $c_\sigma$ is a sequence of $\tpbf{\Sigma}^0_1$-codes of elements of $(\comU_{\sigma,n},\comr_{\sigma\fr n})_{n\in\om}$.
\item If $\sigma$ is labeled by $\phi_\alpha$, then $c_\sigma$ is a $\tpbf{\Sigma}^0_{1+\om^\alpha}$-code of $\comr_{\sigma\fr 0}$.
\end{enumerate}

As before, $(c_\sigma)_{\sigma\in{\tt S}_t}$ can be considered as an element of $\om^\om$.

%A {\em $\tpbf{\Sigma}_t$-code} of a $\tpbf{\Sigma}_t$-function $f$ is a code of a command $\theta$ on $t$ such that $f=[\theta]$.
%We use the symbol $\Sigma_t(\Delta^1_1)$ to denote the set of all $\tpbf{\Sigma}_t$-functions having $\Delta^1_1$ $\tpbf{\Sigma}_t$-codes.

\section{Reassignment-elimination}

\subsection{Flowchart}\label{sec:flowchart}
The working mechanism of a command is a bit unclear, since it involves reassignment instructions ${\tt x}\leftarrow u({\tt x})$ by $\tpbf{\Sigma}^0_\alpha$-measurable functions $u$.
Fortunately, however, it is possible to eliminate the reassignment instructions:
Let us introduce the notion of a flowchart, which is a mechanism for defining functions using only conditional branching.
This notion is particularly useful when dealing with higher-dimensional spaces.
A similar concept has been studied, e.g.~in Kihara \cite{Kih16} and Selivanov \cite{Sel19b}.

\begin{definition}
A {\em flowchart} on an $\lang(Q)$-term $t$ over a topological space $X$ is a family $\bfS=(S_\sigma)_{\sigma\in\syn_t}$ indexed by the syntax tree $\syn_t$ satisfying the following condition:
\begin{enumerate}
\item For a leaf $\sigma\in\syn_t$, then $S_\sigma=X$.
\item If $\sigma\in\syn_t$ is labeled by $\mul$, then $S_\sigma$ is a $\tpbf{\Sigma}^0_{{\rm rank}(\sigma)}$ subset of $X$.
\item If $\sigma\in\syn_t$ is labeled by $\ad$, then $S_\sigma$ is a sequence $(S_{\sigma,n})_{n\in\om}$ of $\tpbf{\Sigma}^0_{{\rm rank}(\sigma)}$ subsets of $X$.
\item If $\sigma\in\syn_t$ is labeled by $\phi_\alpha$, then $S_\sigma=X$.
\end{enumerate}
%We say that a flowchart $\eta$ is {\em strongly total} if $(S_{\sigma,n})_{n\in\om}$ is an open cover of $X$ whenever $\sigma$ is labeled by $\ad$.
\end{definition}

We introduce the notion of a {\em true position} (with respect to a flowchart $\bfS$) in the following inductive manner:
\begin{itemize}
\item The root of the syntax tree $\syn_t$ is a true position for $x$.
\item Assume that $\sigma\in\syn_t$ is a true position for $x$, and is labeled by $\mul$.
Then, if $x\in S_\sigma$ then $\sigma\fr 1$ is a true position for $x$, and if $x\not\in S_\sigma$ then $\sigma\fr 0$ is a true position for $x$.
\item Assume that $\sigma\in\syn_t$ is a true position for $x$, and is labeled by $\ad$.
If $x\in S_{\sigma,n}$ then $\sigma\fr n$ is a true position for $x$.
\item Assume that $\sigma\in\syn_t$ is a true position for $x$, and is labeled by $\phi_\alpha$.
Then, the unique immediate successor $\sigma\fr 0$ of $\sigma$ is a true position for $x$.
\end{itemize}

If a leaf $\rho\in\syn_t$ is a true position for $x$, then we call it a {\em true path for $x$} (with respect to $\bfS$).
A label of a true path for $x$ is considered as an output of the function determined by the flowchart $\bfS$.

\begin{notation}
We write $\eval{\bfS}(x)=q$ if $q$ is a label of a true path for $x$ with respect to $\bfS$.
\end{notation}

As before, it may happen that there are many true paths for $x$ or there is no true path for $x$.
If, for any $x\in X$, there is a true path for $x$, that is, $\eval{\bfS}(x)$ is defined, then we call $\bfS$ {\em total}.
%However, it may happen that there are many true paths for $x$.
If, for any $x\in X$, there are at most one value $q$ such that $\eval{\bfS}(x)=q$, then $\bfS$ is called a {\em deterministic flowchart}.
Then, the function $\eval{\bfS}\colon X\to Q$ is called the {\em evaluation} of the flowchart $\bfS$.
%Note that if $[\eta](x)$ has an output $q\in Q$, then there is a true path $\rho$ for $x$ with respect to $\eta$ such that $\rho$ is labeled by $q$.
%Indeed, consider the set $L_t$ of all leaves of the syntax tree $\syn_t$, and then $D=\bigcup_{\rho\in L_t}D_\rho$ is the domain of the function $[\eta]$ in the sense that $[\eta](x)$ has an output whenever $x\in D$.
%Moreover, if $\eta$ is single-valued and $x\in D_\rho$, then the value $[\eta](x)$ is equal to the label of $\rho$.

\begin{definition}[Fourth definition of $\tpbf{\Sigma}_t$]
Let $X$ be a topological space.
For any $\lang(Q)$-term $t$, define $\tpbf{\Sigma}_t(X)$ as the set of functions of the evaluations $\eval{\bfS}$ of deterministic total flowcharts $\bfS$ on the term $t$ over the space $X$.
\end{definition}

As mentioned above, this notion is easy to handle even in higher-dimensional spaces, and therefore, in this article, we declare this to be the correct definition of $\tpbf{\Sigma}_t$.
As before, one may consider $\eval{\bfS}$ even if $\bfS$ is non-deterministic, and in such a case, $\eval{\bfS}$ is multi-valued.
In this case, we also say that a multi-valued function $g$ is $\tpbf{\Sigma}_t$ if $g$ coincides with $\eval{\bfS}$ for a flowchart $\bfS$ on $t$.

\subsubsection*{Coding a flowchart:}

If $X$ is second-countable, using the standard codings of the $\tpbf{\Sigma}^0_\xi$ sets, one can define the notion of a code of a flowchart $\bfS$ on an $\lang(Q)$-term $t$ in the straightforward manner.
More precisely, a {\em code} of a flowchart $\bfS=(S_\sigma)_{\sigma\in\syn_t}$ is a pair of a code of the syntax tree $\syn_t=({\tt S}_t,{\tt l}_t)$ and an ${\tt S}_t$-indexed collection $(c_\sigma)_{\sigma\in{\tt S}_t}$ of elements of $\om^\om$ satisfying the following conditions:
\begin{enumerate}
%\item If $\sigma$ is labeled by a constant symbol $q$, then $c_\sigma$ is a code of $q$, i.e., $q=q_{c_\sigma}$.
\item If $\sigma$ is a leaf, then $c_\sigma$ is arbitrary.
\item If $\sigma$ is labeled by $\mul$, then $c_\sigma$ is a $\tpbf{\Sigma}^0_{\rank(\sigma)}$-code of $S_\sigma$.
\item If $\sigma$ is labeled by $\ad$, then $c_\sigma$ is a sequence of $\tpbf{\Sigma}^0_{\rank(\sigma)}$-codes of elements of $(S_{\sigma,n})_{n\in\om}$.
\item If $\sigma$ is labeled by $\phi_\alpha$, then $c_\sigma$ is arbitrary.
\end{enumerate}

As before, $(c_\sigma)_{\sigma\in{\tt S}_t}$ can be considered as an element of $\om^\om$.
Since there is a total representation of $\tpbf{\Sigma}^0_\xi$ sets for each $\xi<\om_1$, any sequence $(c_\sigma)_{\sigma\in{\tt S}_t}$ can be considered as a code of a flowchart.
If $t$ is a hyperarithmetical $\lang(Q)$-term, we say that a flowchart is $\Delta^1_1$ if it has a $\Delta^1_1$ code.
Define $\Sigma_t(\Delta^1_1;X)$ as the set of all functions $g\colon X\to Q$ determined by a $\Delta^1_1$ flowchart.
If the underlying space is clear from the context, we simply write $\Sigma_t(\Delta^1_1)$ instead of $\Sigma_t(\Delta^1_1;X)$.
We will discuss the complexity of the codes of total and deterministic flowcharts in Lemma \ref{lem:complexity-flowchart}.
%For a set $A\subseteq X$, let ${\tt Code}_t(A)$ be the set of all codes of flowcharts $\bfS$ over $X$ such that the domain of $\eval{\bfS}$ includes $A$.

\subsection{Technical lemmata}

We introduce here some useful concepts on flowcharts that we will use later, but not immediately.
First, one can define the notion of a true position by assigning a set $D_\sigma\subseteq X$ to each $\sigma\in\syn_t$.
\begin{enumerate}
\item For the root $\langle\rangle$ of $\syn_t$, define $D_{\langle\rangle}=X$.
\item If $\sigma$ is labeled by $\mul$, then define $D_{\sigma\fr 0}=D_\sigma\setminus S_\sigma$ and $D_{\sigma\fr 1}=D_\sigma\cap S_\sigma$.
\item If $\sigma$ is labeled by $\ad$, then define $D_{\sigma\fr n}=D_\sigma\cap S_{\sigma,n}$.
\item If $\sigma$ is labeled by $\phi_\alpha$, then define $D_{\sigma\fr 0}=D_\sigma$.
\end{enumerate}

We call $(D_\sigma)_{\sigma\in\syn_t}$ the {\em domain assignment} to $\bfS=(S_\sigma)_{\sigma\in\syn_t}$.
It is easy to see that $\sigma\in\syn_t$ is a true position for $x$ with respect to $\bfS$ if and only if $x\in D_\sigma$.
Hence, $\bfS$ is total if and only if $(S_{\sigma,n})_{n\in\om}$ covers $D_\sigma$ whenever $\sigma$ is labeled by $\ad$.
Using this notion, let us show a few complexity results on flowcharts.

\begin{lemma}\label{lem:complexity-flowchart}
Let $X$ be a computable Polish space, and $t$ be a hyperarithmetical $\lang(Q)$-term.
\begin{enumerate}
\item If $g\pcolon X\to Q$ is $\Sigma_t(\Delta^1_1)$, then there exists a $\Delta^1_1$-measurable function $G\pcolon X\to\om^\om$ such that $g(x)=q_{G(x)}$ for any $x\in{\rm dom}(g)$, where $Q$ is indexed as $\{q_z\}_{z\in\Baire}$.
\item The set of all codes of deterministic flowcharts on $t$ over $X$ is $\Pi^1_1$.
\item If $A$ is a $\Sigma^1_1$ subset of $X$, then the set of all codes of flowcharts $\bfS$ on $t$ over $X$ such that the domain of $\eval{\bfS}$ includes $A$ is $\Pi^1_1$.
In particular, the set of all codes of total flowcharts on $t$ over $X$ is $\Pi^1_1$.
\end{enumerate}
\end{lemma}

\begin{proof}
(1)
Given a code $c$ of a flowchart $\bfS$, one can easily see that the domain assignment $(D_\sigma)_{\sigma\in\syn_t}$ for $\bfS$ is uniformly $\Delta^1_1$ relative to the code $c$.
Moreover, a node $\sigma\in\syn_t$ is a true position for $x$ if and only if $x\in\bigcap_{\tau\preceq\sigma}D_\tau$.
This is a $\Delta^1_1$ property relative to $c$.
Moreover, recall that a code of a syntax tree contains the information about what each node of the tree is labeled.
In particular, one can effectively recover information about the label of $\sigma$.
Let $G(x)$ return a code of the label of a true path for $x$, and then $G$ is clearly $\Delta^1_1$-measurable relative to $c$.

(2)
By definition, $c$ is a code of a deterministic flowchart if and only if for any $x\in X$ and any $\sigma,\tau\in\syn_t$, if both $\sigma$ and $\tau$ are true paths for $x$ w.r.t.~the flowchart coded by $c$, then both $\sigma$ and $\tau$ are labeled by the same symbol.
This is a $\Pi^1_1$ property.

(3) One can see that the domain of $\eval{\bfS}$ includes $A$ if and only if, whenever $\sigma$ is labeled by $\ad$, the sequence $(S_{\sigma,n})_{n\in\om}$ covers $A\cap D_\sigma$, i.e., for any $x\in X$, if $x\in A\cap D_\sigma$ then $x\in S_{\sigma,n}$ for some $n\in\om$.
Since $A$ is $\Sigma^1_1$, and $D_\sigma$ and $(S_{\sigma,n})_{n\in\om}$ are $\Delta^1_1$ relative to a given code, this is a $\Pi^1_1$ property.
\end{proof}

We say that a flowchart $\bfS$ is {\em monotone} if any set assigned to $\sigma\in \syn_t$ by $\bfS$ is a subset of $D_\sigma$.
In other words, if $\sigma$ is labeled by $\mul$, then $D_{\sigma\fr 1}=S_\sigma$; and if $\sigma$ is labeled by $\ad$, then $D_{\sigma\fr n}=S_{\sigma,n}$.
%\begin{itemize}
%\item If $\sigma$ is labeled by $\mul$, then any set assigned to $\tau\succeq\sigma\fr 1$ is a subset of $S_\sigma$.
%\item If $\sigma$ is labeled by $\ad$, then any set assigned to $\tau\succeq\sigma\fr n$ is a subset of $S_{\sigma,n}$.
%\end{itemize}

\begin{lemma}[see also Selivanov \cite{Sel19b}]\label{lem:monotone}
Given a flowchart $\bfS$ on a normal $\lang(Q)$-term $t$ one can effectively find a monotone flowchart $\bfS'$ on $t$ determining the same function as $\bfS$.
\end{lemma}

\begin{proof}
Let $(D_\sigma)_{\sigma\in\syn_t}$ be the domain assignment for $\bfS$.
We inductively show that, if $\sigma\in\syn_t$ is not a leaf, then $D_\sigma$ is a $\tpbf{\Sigma}_{\rank(\sigma)}$ subset of $X$.
If $\sigma$ is labeled by $\ad$, then $D_{\sigma\fr n}=D_\sigma\cap S_{\sigma,n}$ is a $\tpbf{\Sigma}_{\rank(\sigma)}$ subset of $X$ since both $D_\sigma$ and $S_{\sigma,n}$ are $\tpbf{\Sigma}_{\rank(\sigma)}$ by induction hypothesis.
Since $\rank(\sigma)\leq\rank(\sigma\fr n)$, we get $D_{\sigma\fr n}\in \tpbf{\Sigma}_{\rank(\sigma\fr n)}$.
If $\sigma$ is labeled by $\mul$, then $D_{\sigma\fr 1}=D_\sigma\cap S_{\sigma,1}$ is $\tpbf{\Sigma}_{\rank(\sigma\fr n)}$ as above.
One can also see that $D_{\sigma\fr 0}=D_\sigma\setminus S_{\sigma}$ is a $\tpbf{\Sigma}_{\rank(\sigma)+1}$ subset of $X$.
By normality of $t$, $\sigma\fr 0$ is either a leaf or labeled by $\phi_\alpha$.
If $\sigma\fr 0$ is a leaf, there is nothing to do.
If $\sigma\fr 0$ is labeled by $\phi_\alpha$, we have $\rank(\sigma)+1\leq\rank(\sigma\fr 0)$, and therefore, we get $D_{\sigma\fr 0}\in\tpbf{\Sigma}_{\rank(\sigma\fr 0)}$.

If $\sigma$ is labeled by $\mul$, define $S'_\sigma=D_\sigma\cap S_\sigma$, which is $\tpbf{\Sigma}^0_{\rank(\sigma)}$; and if $\sigma$ is labeled by $\ad$, define $S'_{\sigma,n}=D_\sigma\cap S_{\sigma,n}$, which is $\tpbf{\Sigma}^0_{\rank(\sigma)}$.
Hence, $(\bfS'_\sigma)_{\sigma\in\syn_t}$ gives a monotone flowchart on $t$, which clearly satisfies $\eval{\bfS}=\eval{\bfS'}$.
\end{proof}

We next see that, if the space $X$ is zero-dimensional, then for any node $\sigma\in\syn_t$ labeled by $\sqcup$, the assigned sequence $(U_{\sigma,n})_{n\in\om}$ can be pairwise disjoint.
Note that this property is required in the original definition of $\tpbf{\Sigma}_t$ (Definition \ref{def:K-M-Sigma2}).
Such a flowchart may be called a reduced flowchart, following the terminology in \cite{Sel20}.
This property is particularly useful when combining more than one deterministic flowcharts to make a new deterministic flowchart; see also the proof of Theorem \ref{thm:decomposition}.
%We will not use this property in this paper since our definition differs from the original one; however, it is necessary if we want to construct a partition rather than a covering, as in the original definition.

\begin{prop}\label{prop:zero-dim-partition}
Let $X$ be a subspace of a zero-dimensional Polish space, and let $\bfS$ be a flowchart on an $\lang(Q)$-term $t$ over $X$.
Then, in a hyperarithmetical manner, one can find a flowchart $\bfS'$ on $t$ such that for any node $\sigma\in\syn_t$ labeled by $\sqcup$, the assigned sequence $(U_{\sigma,n})_{n\in\om}$ is pairwise disjoint, and $\bfS'$ determines the same function as $\bfS$. 
\end{prop}

\begin{proof}
It suffices to show that, for any $\xi<\om_1^{\rm CK}$, given a uniform $\Delta^1_1$-sequence of $\tpbf{\Sigma}^0_\xi$ sets $(R_n)_{n\in\om}$, one can effectively find a $\Delta^1_1$-sequence of pairwise disjoint $\tpbf{\Sigma}^0_\xi$ sets such that $R^\ast_n\subseteq R_n$ for each $n\in\om$ and $(R_n^\ast)_{n\in\om}$ and $\bigcup_{n\in\om}R_n^\ast=\bigcup_{n\in\om}R_n$.
To see this, let $R$ be the $\tpbf{\Sigma}^0_\xi(\Delta^1_1)$ set such that $(x, n)\in R$ if and only if $x\in R_n$.
If $P$ is a universal $\Sigma^0_\xi$ set, there exists a $\Delta^1_1$ element $\ep\in\Baire$ such that $(x,n)\in R$ if and only if $(\ep,x,n)\in P$.
By the uniformization property on $\Sigma^0_\xi(\ep)$ (see \cite[3E.10]{Mos09}), we have some $R^\ast\in\Sigma^0_\xi(\ep)$ which uniformizes $R$.
Now, let $R^\ast_n$ be the $\Sigma^0_\xi(\ep)$ set such that $x\in R^\ast_n$ if and only if $(x,n)\in R^\ast$.
Then, $(R^\ast_n)_{n\in \omega}$ has the required property.
\end{proof}

\subsection{Translation}

Now, for an $\lang(Q)$-term $t$, we will show that the four definitions of $\tpbf{\Sigma}_t$ introduced so far all define the same class of functions.
\[\tpbf{\Sigma}_t(\Baire)=\tpbf{\Sigma}^W_t(\Baire)=\tpbf{\Sigma}^{W\circ}_t(\Baire)=\tpbf{\Sigma}'_t(\Baire).\]

Note that the equivalence $\tpbf{\Sigma}_t(\Baire)=\tpbf{\Sigma}_t'(\Baire)$ has also been shown in \cite[Theorem 4.10]{Sel20}, although the terminology is slightly different.
To prove the equivalence, it suffices to show the following:

\begin{theorem}\label{prop:matryo-flow}
Let $t$ be an $\lang(Q)$-term, and assume that $Z=\om^\om$.
\begin{enumerate}
\item Given a command $\bfU$ on $t$ over $Z$, one can effectively find a flowchart $\bfS$ on $t$ over $Z$ such that $\eval{\bfU}=\eval{\bfS}$.
%Moreover, if $\theta$ is total, so is $\eta$, and if $\theta$ is deterministic, so is $\eta$.
\item Conversely, given a flowchart $\bfS$ on $t$ over $Z$, one can effectively find a simple command $\bfU$ on $t$ over $Z$ such that $\eval{\bfS}=\eval{\bfU}$.
If $\bfS$ is total, then $\bfU$ is total and simple, and moreover one can effectively find a strongly total command $\bfU'$ on $t$ over $Z$ such that $\eval{\bfS}=\eval{\bfU'}$. 
\end{enumerate}
\end{theorem}

\begin{proof}
(1) Given a command $\bfU=(\comU_\sigma,\comr_\sigma)_{\sigma\in\syn_t}$, we construct a flowchart $\bfS=(S_\sigma)_{\sigma\in\syn_t}$ as follows:
\begin{itemize}
\item For a leaf $\sigma\in\syn_t$, then $S_\sigma=Z$.
\item If $\sigma\in\syn_t$ is labeled by $\mul$, then define $S_\sigma=\val_\sigma^{-1}[U_\sigma]$.
\item If $\sigma\in\syn_t$ is labeled by $\ad$, then define $S_{\sigma,n}=(\val_\sigma^{-1}[U_{\sigma,n}])_{n\in\om}$.
\item If $\sigma\in\syn_t$ is labeled by $\phi_\alpha$, then $S_\sigma=Z$.
\end{itemize}

Then, $S_\sigma$ and $S_{\sigma,n}$ are $\tpbf{\Sigma}^0_{{\rm rank}(\sigma)}$ subsets of $Z$ since ${\tt val}_\sigma$ is a $\tpbf{\Sigma}^0_{{\rm rank}(\sigma)}$-measurable function with a $\tpbf{\Sigma}^0_{{\rm rank}(\sigma)}$ domain by Observation \ref{obs:val-measurable}.
%If $\sigma$ is labeled by $\ad$, $(S_{\sigma,n})_{n\in\om}$ covers $X$ since $(U_{\sigma,n})_{n\in\om}$ covers $X$ and the domain of $\val_\sigma$ is $X$.
Hence, $\bfS$ is a flowchart on the $\lang(Q)$-term $t$.
Then $\bfU$ and $\bfS$ determine the same function, i.e., $\eval{\bfU}=\eval{\bfS}$.
This is because, by definition, if $\sigma$ is labeled by $\mul$, then $x\in S_\sigma$ if and only if $\val_\sigma(x)\in U_\sigma$; and if $\sigma$ is labeled by $\ad$, then $x\in S_{\sigma,n}$ if and only if $\val_\sigma(x)\in U_{\sigma,n}$.
By induction, this trivially ensures that $\sigma$ is a true position for $x$ with respect to $\bfS$ if and only if $\sigma$ is a true position for $x$ with respect to $\bfU$.
In particular, $\sigma\in\syn_t$ is a true path for $x$ with respect to $\bfS$ if and only if $\sigma$ is a true path for $x$ with respect to $\bfU$.
This means that $\eval{\bfS}(x)=q$ if and only if $\eval{\bfU}(x)=q$.

\medskip

(2)
First, we fix a sufficiently strong oracle $\delta$, relative to which all the Borel sets attached to the given flowchart $\bfS$ be lightface.
For any $\lang(Q)$-term $t$, its syntax tree is countable, and therefore, only countably many Borel sets are attached; hence such an oracle $\delta$ exists.
Then, for any $\delta$-computable ordinal $\alpha$, let us consider the $\om^\alpha$-th Turing jump operator $j_\alpha:=\mathcal{J}^{\om^\alpha,\delta}\colon Z\to Z$ with true stages relative to $\delta$; see Kihara-Montalb\'an \cite[Sections 4.1 and 6.1]{KM19}.
For the sake of brevity, hereafter, $\delta$ will be omitted from the notation.
One of the key properties of the jump operator $j_\alpha$ with true stages is that its image $j_\alpha[Z]$ is closed.
Moreover, a single index ensures the inequality $x\leq_Tj_\alpha(x)$ for any $x$, and therefore, $j_\alpha$ has a computable left-inverse $j_\alpha^{-1}\colon j_\alpha[Z]\to Z$.
By combining these properties, observe that, if $A\subseteq Z$ is closed, so is $j_\alpha[A]$.

\medskip

{\bf Step 1.}
We first transform a flowchart $\bfS=(S_\sigma)_{\sigma\in\syn_t}$ into a command $\bfU$.
First we define $(u_\sigma)_{\sigma\in\syn_t}$ as follows:
If $\sigma\in\syn_t$ is labeled by $\mul$ or $\ad$ then $u_{\sigma\fr n}={\rm id}$ for each $n$; and if $\sigma$ is labeled by $\phi_\alpha$ then $u_{\sigma\fr 0}=j_\alpha$.
Then we define a partial function $\val_\sigma$ as before.
Let $(\tau_i)_{i<\ell}$ be the Veblen initial segments of a node $\sigma\in\syn_t$, where $\tau_i$ is labeled by $\phi_{\alpha_i}$.
Then, note that
\[\val_\sigma(x)=j_{\alpha_\ell}\circ j_{\alpha_{\ell-1}}\circ\dots\circ j_{\alpha_1}\circ j_{\alpha_0}(x).\]

Next, we define $(U_\sigma)_{\sigma\in\syn_t}$ as follows:
If $\sigma$ is labeled by $\ad$, since $S_\sigma$ is a $\bS_{\rank(\sigma)}$ set, by universality of jump operators, one can effectively find an open set $U_\sigma$ such that
\[S_\sigma=(j_{\alpha_\ell}\circ j_{\alpha_{\ell-1}}\circ\dots\circ j_{\alpha_1}\circ j_{\alpha_0})^{-1}[U_\sigma].\]

Similarly, if $\sigma$ is labeled by $\mul$, for any $n$, one can effectively find an open set $U_{\sigma,n}$ such that $S_{\sigma,n}=(j_{\alpha_\ell}\circ j_{\alpha_{\ell-1}}\circ\dots\circ j_{\alpha_1}\circ j_{\alpha_0})^{-1}[U_{\sigma,n}]$.
If $\sigma$ is labeled by $\phi_\alpha$, then put $U_\sigma=Z$.
Then, $\bfU=(U_\sigma,u_\sigma)_{\sigma\in\syn_t}$ gives a simple command on the $\lang(Q)$-term $t$.

\begin{claim}\label{claim:transform1}
$\bfS$ and $\bfU$ determine the same function, i.e., $\eval{\bfS}=\eval{\bfU}$.
\end{claim}

\begin{proof}
By definition, if $\sigma$ is labeled by $\mul$, then $x\in S_\sigma$ if and only if $\val_\sigma(x)\in U_\sigma$; and if $\sigma$ is labeled by $\ad$, then $x\in S_{\sigma,n}$ if and only if $\val_\sigma(x)\in U_{\sigma,n}$.
Therefore, as in the proof of the item (1), one can easily see that $\eval{\bfS}(x)=q$ if and only if $\eval{\bfU}(x)=q$.
\end{proof}

\begin{claim}
If $\bfS$ is total, so is $\bfU$.
\end{claim}

\begin{proof}
Assume that $\sigma$ is a true position for $x$ with respect to $\bfS$, and $\sigma$ is labeled by $\ad$.
By the proof of Claim \ref{claim:transform1}, this implies that $\sigma$ is also a true position for $x$ with respect to $\bfU$.
Since $\bfS$ is total, there exists $n\in\om$ such that $\sigma\fr n$ is a true position for $x$ with respect to $\bfS$.
This means that $x\in S_{\sigma,n}$ for some $n\in\om$.
By definition, $x\in S_{\sigma,n}$ if and only if $\val_\sigma(x)\in U_{\sigma,n}$.
Thus, $\sigma\fr n$ is a true position for $x$ with respect to $\bfU$.
By induction, this shows that $\bfU$ is total.
\end{proof}

{\bf Step 2.}
Next we transform the above total commend $\bfU$ into a strongly total command $\bfU'$.
We define the domain assignment of the command $\bfU=(U_\sigma)_{\sigma\in\syn_t}$ as follows:
\begin{enumerate}
\item For the root $\langle\rangle$ of $\syn_t$, define $D_{\langle\rangle}=Z$.
\item If $\sigma$ is labeled by $\mul$, then define $D_{\sigma\fr 0}=D_\sigma$ and $D_{\sigma\fr 1}=D_\sigma\cap U_\sigma$.
\item If $\sigma$ is labeled by $\ad$, then define $D_{\sigma\fr n}=D_\sigma\cap U_{\sigma,n}$.
\item If $\sigma$ is labeled by $\phi_\alpha$, then define $D_{\sigma\fr 0}=j_\alpha[D_\sigma]$.
\end{enumerate}

Note that we always have $\val_\sigma(x)\in D_\sigma$ whenever $\sigma$ is a true position for $x$ with respect to $\bfU$.
We will inductively define a computable homeomorphism $\iota_\sigma\colon D_\sigma\simeq Z_\sigma$, where $Z_\sigma$ is a closed subset of $Z$.

To describe our construction, first we note that the Turing jump operator has the so-called uniformly order preserving (UOP) property:
There exists a computable function $p\colon\om\to\om$ such that
\[x\leq_Ty\mbox{ via $e$}\implies j_\alpha(x)\leq_Tj_\alpha(y)\mbox{ via $p(e)$}.\]

Assume that $\iota\colon A\simeq B$ is a computable homeomorphism, and let $(d,e)$ be a pair of indices of $\iota$ and $\iota^{-1}$.
Then, by the UOP property of the jump, the pair $(p(d),p(e))$ gives a computable homeomorphism $\iota^\star\colon j_\alpha[A]\simeq j_\alpha[B]$ such that $j_\alpha\circ\iota(x)=\iota^\star\circ j_\alpha(x)$ for any $x\in A$.

%We also define $(Z_\sigma)$ as follows:
%\begin{enumerate}
%\item For the root $\langle\rangle$ of $\syn_t$, define $Z_{\langle\rangle}=Z$.
%\item If $\sigma$ is labeled by $\mul$, then define $Z_{\sigma\fr 0}=Z$ and $Z_{\sigma\fr 1}=Z$.
%\item If $\sigma$ is labeled by $\ad$, then define $Z_{\sigma\fr n}=Z$.
%\item If $\sigma$ is labeled by $\phi_\alpha$, then define $Z_{\sigma\fr 0}=j_\alpha[Z_\sigma]$.
%\end{enumerate}

We also recall that, for an open set $U\subseteq Z$, ${\tt in}_U\colon Z\simeq U$ denotes a homeomorphism between $Z$ and $U$.
For a function $f\colon A\to B$ and $C\subseteq B$, we use the symbol $f\coupto C$ to denote $f\upto f^{-1}[C]$, i.e., the restriction of $f$ whose codomain is $C$.
For a subspace $Y$ of $Z$, if $U$ is open in $Y$ then there is an open set $\hat{U}$ in $Z$ such that $U=\hat{U}\cap Y$.
Then, the homeomorphism ${\tt in}_{\hat{U}}\colon Z\simeq\hat{U}$ induces another homeomorphism ${\tt in}_{\hat{U}}\coupto Y\colon {\tt in}_{\hat{U}}^{-1}[Y]\simeq U$.
Note that if $Y$ is closed in $Z$, so is the domain ${\tt in}_{\hat{U}}^{-1}[Y]$.
In the following argument, we may think of ${\tt in}_{\hat{U}}^{-1}$ as a magnifying glass which enlarges $\hat{U}$ to the size of the whole space $Z$, and ${\tt in}_{\hat{U}}^{-1}[Y]$ as the view of $U=\hat{U}\cap Y$ under this scale.

Now, let us start the construction of $\iota_\sigma\colon D_\sigma\simeq Z_\sigma$.
First put $\iota_\ep={\rm id}$ and $Z_\ep=Z$, and assume that $\iota_\sigma$ and $Z_\sigma$ has already been defined.

\medskip

{\bf Case 1.}
If $\sigma$ is labeled by $\mul$, then put $V'_\sigma=\iota_\sigma[D_\sigma\cap U_\sigma]$, and note that, by the induction hypothesis, $i_\sigma\colon D_\sigma\simeq Z_\sigma$ is a homeomorphism, so $V'_\sigma$ is open in $Z_\sigma$ since $U_\sigma$ is open.
Then, there exists an open set $\hat{V}'_\sigma$ in $Z$ such that $V'_\sigma=\hat{V}'_\sigma\cap Z_\sigma$.
Then put
\begin{align*}
Z_{\sigma\fr 1}={\tt in}^{-1}_{\hat{V}'_\sigma}[Z_\sigma],
& &
u'_{\sigma\fr 1}=\left({\tt in}_{\hat{V}'_\sigma}\coupto Z_\sigma\right)^{-1}\colon V'_\sigma\simeq Z_{\sigma\fr 1}.
\end{align*}

As mentioned above, $Z_{\sigma\fr 1}$ is closed in $Z$.
Then define $\iota_{\sigma\fr 1}$ as the restriction of $u'_{\sigma\fr 1}\circ\iota_\sigma$ up to $U_\sigma$.
Since $D_{\sigma\fr 1}=D_\sigma\cap U_\sigma$, the map $\iota_{\sigma\fr 1}\colon D_{\sigma\fr 1}\simeq Z_{\sigma\fr 1}$ is a homeomorphism.
Diagrammatically, this argument may be described as follows:
\[
\xymatrix{
D_\sigma \ar[rr]^{\sim}_{\iota_\sigma}  & & Z_\sigma & &  \\
D_\sigma\cap U_\sigma \ar[rr]^{\sim}_{\iota_\sigma\upto U_\sigma} \ar@{^{(}->}[u] & & \hat{V}_\sigma'\cap Z_\sigma  \ar@{^{(}->}[d] \ar@{^{(}->}[u]  & & Z_{\sigma\fr 1} \ar[ll]^{{\tt in}_{\hat{V}_\sigma'}\coupto Z_\sigma}_{\sim}  \ar@{^{(}->}[d] \\
& & \hat{V}_\sigma' & & Z \ar[ll]^{{\tt in}_{\hat{V}_\sigma'}}_{\sim}
}
\]

We also put $u'_{\sigma\fr 0}={\rm id}$, $Z_{\sigma\fr 0}=Z_\sigma\setminus\hat{V}'_\sigma$, and $\iota_{\sigma\fr 0}=\iota_\sigma\upto D_\sigma\fr 0$.
Note that $\iota_\sigma[D_{\sigma\fr 0}]=\iota[D_\sigma\setminus U_\sigma]=Z_\sigma\setminus V_\sigma'=Z_{\sigma\fr 0}$.
Hence, $\iota_{\sigma\fr 0}\colon D_{\sigma\fr 0}\simeq Z_{\sigma\fr 0}$ is a homeomorphism.

\medskip

{\bf Case 2.}
If $\sigma$ is labeled by $\ad$, then put $V'_{\sigma,n}=\iota_\sigma[D_\sigma\cap U_{\sigma,n}]$, 
and let $\hat{V}'_{\sigma,n}$ be an open set in $Z$ such that $V'_{\sigma,n}=\hat{V}'_{\sigma,n}\cap Z_\sigma$ for each $n\in\om$, as above.
Then put
\begin{align*}
Z_{\sigma\fr n}={\tt in}^{-1}_{\hat{V}'_{\sigma,n}}[Z_\sigma],
& &
u'_{\sigma\fr n}=\left({\tt in}_{\hat{V}'_{\sigma,n}}\coupto Z_\sigma\right)^{-1}\colon V'_{\sigma,n}\simeq Z_{\sigma\fr n}.
\end{align*}

As before $Z_{\sigma\fr n}$ is closed in $Z$.
Then define $\iota_{\sigma\fr n}$ as the restriction of $u'_{\sigma\fr n}\circ\iota_\sigma$ up to $U_{\sigma,n}$.
Since $D_{\sigma\fr n}=D_\sigma\cap U_{\sigma,n}$, the map $\iota_{\sigma\fr n}\colon D_{\sigma\fr n}\simeq Z_{\sigma\fr n}$ is a homeomorphism.

\medskip

{\bf Case 3.}
If $\sigma$ is labeled by $\phi_\alpha$, then put $\hat{V}'_\sigma=Z$, $Z_{\sigma\fr 0}=j_\alpha[Z_\sigma]$, $u'_{\sigma\fr 0}=j_\alpha$.
By the property of our specific Turing jump operator mentioned above, $Z_{\sigma\fr 0}$ is closed in $Z$.
Since $\iota_\sigma\colon D_\sigma\simeq Z_\sigma$ by the induction hypothesis, one can effectively find a homeomorphism $\iota^\star\colon j_\alpha[D_\sigma]\simeq j_\alpha[Z_\sigma]$ by the UOP property.
Then define $\iota_{\sigma\fr 0}=\iota_\sigma^\star$.
Diagrammatically,
\[
\xymatrix{
D_\sigma \ar[rr]^{\sim}_{\iota_\sigma} \ar[d]_{j_\alpha} & & Z_\sigma \ar[d]^{j_\alpha} \\
j_\alpha[D_\sigma] \ar[rr]^{\sim}_{\iota_\sigma^\star} & & j_\alpha[Z_\sigma]
}
\]

Finally, if $\sigma\in\syn_t$ is labeled by $\mul$ or $\phi_\alpha$, we define $U'_\sigma=\hat{V}'_\sigma$; and if $\sigma$ is labeled by $\ad$, we define $U'_{\sigma,0}=\hat{V}'_{\sigma,0}\cup(Z\setminus Z_\sigma)$ and $U'_{\sigma,n}=\hat{V}'_{\sigma,n}$ for each $n>0$.
Note that $U'_{\sigma,0}$ is open since $Z_\sigma$ is closed.
Let us consider the command $\bfU'=(U'_\sigma,u'_\sigma)_{\sigma\in\syn_t}$.

\begin{claim}
The command $\bfU'$ is strongly total.
\end{claim}

\begin{proof}
By the property of the domain assignment,
%$\sigma$ is true position for $x$ if and only if $\val_\sigma(x)\in D_\sigma$.
if $\sigma$ is labeled by $\ad$, then $D_\sigma$ is covered by $(U_{\sigma,n})_{n\in\om}$ since the command $\bfU$ is total.
%By the previous claim, $\val'_\sigma$ agrees with $\iota_\sigma\circ\val_\sigma$ for such an $x$, and since $Z_\sigma=\iota[D_\sigma]$, if $\val_\sigma(x)\in D_\sigma$ then $\val'_\sigma(x)\in Z_\sigma$.
%Since $(U_{\sigma,n})_{n\in\om}$ covers $D_\sigma$, we have
Therefore,
\[Z_\sigma=\iota[D_\sigma]\subseteq\iota_\sigma\left[D_\sigma\cap\bigcup_{n\in\om}U_{\sigma,n}\right]=\bigcup_{n\in\om}\iota_\sigma[D_\sigma\cap U_{\sigma,n}]=\bigcup_{n\in\om}V'_{\sigma,n}\subseteq\bigcup_{n\in\om}U_{\sigma,n}'.\]

Moreover, we have $Z\setminus Z_\sigma\subseteq U_{\sigma,0}'$.
Hence, $(U'_{\sigma,n})_{n\in\om}$ is a cover of $Z$.
This means that $\bfU'$ is strongly total.
\end{proof}

As before, for a node $\sigma\in\syn_t$ of length $\ell$, we define a partial function $\val'_\sigma$ as follows:
\[\val'_\sigma=u'_\sigma\circ u'_{\sigma\upto(\ell-1)}\circ\dots\circ u'_{\sigma\upto 2}\circ u'_{\sigma\upto 1}.\]

If $\sigma$ is a true position for $x$ with respect to $\bfU'$, then $\val'_\sigma(x)\in Z_\sigma$.

\begin{claim}
The commands $\bfU$ and $\bfU'$ determine the same function, i.e., $\eval{\bfU}=\eval{\bfU'}$.
%For any $x$, if $\val_\sigma(x)\in D_\sigma$, then we have $\val'_\sigma(x)=\iota_\sigma\circ\val_\sigma(x)$.
%Similarly, if $\val_\sigma'(x)\in Z_\sigma$, then we have $\val_\sigma(x)=\iota_\sigma^{-1}\circ\val'_\sigma(x)$.
\end{claim}

\begin{proof}
We inductively show $\val'_\sigma(x)=\iota_\sigma\circ\val_\sigma(x)$ for any $x\in Z$.
If $\sigma=\ep$, the assertion is clear.
If $\sigma$ is labeled by $\ad$, then by the induction hypothesis, we have $\val'_{\sigma\fr n}(x)=u'_{\sigma\fr n}\circ\val'_\sigma(x)=u_{\sigma\fr n}'\circ\iota_\sigma\circ\val_\sigma(x)$.
By definition, we have $\iota_{\sigma\fr n}(z)=u'_{\sigma\fr n}\circ\iota_\sigma(z)$ for any $z\in D_\sigma\cap U_{\sigma,n}$.
Therefore, since $\val_{\sigma\fr n}(x)\in D_{\sigma\fr n}=D_\sigma\cap U_{\sigma,n}$, we have
\[\val'_{\sigma\fr n}(x)=u_{\sigma\fr n}'\circ\iota_\sigma\circ\val_\sigma(x)=\iota_{\sigma\fr n}\circ\val_{\sigma}(x)=\iota_{\sigma\fr n}\circ\val_{\sigma\fr n}(x).\]

The same argument applies when $\sigma$ is labeled by $\mul$.
If $\sigma$ is labeled by $\phi_\alpha$,
%then, $\val_{\sigma\fr 0}(x)=j_\alpha\circ\val_\sigma(x)\in D_{\sigma\fr 0}=j_\alpha[D_\sigma]$.
%Since $j_\alpha$ is injective, this implies that $\val_\sigma(x)\in D_\sigma$.
then by the induction hypothesis and the UOP property,
\[\val'_{\sigma\fr 0}=u'_{\sigma\fr 0}\circ\val'_{\sigma}=j_\alpha\circ\iota_\sigma\circ\val_\sigma=\iota_\sigma^\star\circ j_\alpha\circ\val_\sigma=\iota_{\sigma\fr 0}\circ\val_{\sigma\fr 0}.\]

By induction we next prove that a node $\sigma\in\syn_t$ is a true position for $x$ with respect to $\bfU$ if and only if $\sigma$ is a true position for $x$ with respect to $\bfU'$.
Assume that $\sigma$ is a true position for $x$ with respect to $\bfU$ if and only if $\sigma$ is a true position for $x$ with respect to $\bfU'$.
In this case, $\val_\sigma(x)$ is defined and contained in $D_\sigma$.
If $\sigma$ is labeled by $\ad$, then $\val_\sigma(x)\in D_\sigma\cap U_{\sigma,n}$ if and only if $\val'_\sigma(x)=\iota_\sigma\circ\val_\sigma(x)\in\iota_\sigma[D_\sigma\cap U_{\sigma,n}]=V'_{\sigma,n}$.
Since $U'_{\sigma,n}\cap Z_{\sigma}=V'_{\sigma,n}=Z_\sigma\cap U'_{\sigma,n}$.
Hence, $\sigma\fr n$ is a true position for $x$ with respect to $\bfU$ if and only if $\sigma\fr n$ is a true position for $x$ with respect to $\bfU'$.
The same argument applies when $\sigma$ is labeled by $\mul$ or $\phi_\alpha$.

In particular, $\sigma\in\syn_t$ is a true path for $x$ with respect to $\bfU$ if and only if $\sigma$ is a true path for $x$ with respect to $\bfU'$.
This means that $\eval{\bfU}(x)=q$ if and only if $\eval{\bfU'}(x)=q$.
\end{proof}

Note that the above construction is effective, that is, there exists a computable transformation of a given code of a total flowchart $\bfS$ on the $\lang(Q)$-term $t$ into a code of a strongly total command $\bfU'$ on the term $t$ such that $\eval{\bfS}=\eval{\bfU'}$.
This concludes the proof of Theorem \ref{prop:matryo-flow}.
\end{proof}

\section{Symbolic Wadge reducibility}

\subsection{Symbolic representation}

As an advantage of defining $\tpbf{\Sigma}_t$ by means of a flowchart,
%we observe that the equivalence of the classes $\tpbf{\Sigma}_t$ and the Borel Wadge classes holds even in higher-dimensional spaces.
in this section, we argue that one can export many results on the term classes $\tpbf{\Sigma}_t$ on zero-dimensional Polish spaces to arbitrary Polish spaces.
First, one of the basic ideas in modern computability theory is that various mathematical objects (such as real and complex numbers) can be represented by symbolic sequences.
A symbolic sequence is an element of $I^\om$, where $I$ is an alphabet, i.e., a set of symbols.
In this article, we assume that $I=\om$, and consider the set $Z=\om^\om$ of all symbolic sequences.

A {\em symbolic representation}, or simply a {\em representation}, of a set $X$ is a partial surjection $\delta_X\pcolon Z\to X$, where $Z$ is as above.
If $\delta_X(p)=x$, then $p$ is called a {\em $\delta_X$-name of $x$}\index{name} (or simply, a {\em name of $x$} if $\delta_X$ is clear from the context).
A pair of a set and its representation is called a {\em represented space}.
A representation $\delta\colon Z\to X$ is {\em admissible} if for any partial continuous function $f\pcolon Z\to X$ there exists a continuous function $g\pcolon Z\to Z$ such that $f=\delta\circ g$.
It is known that every second-countable $T_0$ space (indeed, every $T_0$ space having a countable cs-network) has an admissible representation; see Schr\"oder \cite{Sch02}.

\begin{example}
Let $\mathcal{X}=(X,d,\alpha)$ be a separable metric space, where $\{\alpha_i\}_{i\in\om}$ is a dense subset of $X$.
Then, a {\em Cauchy name}\index{Cauchy name} of a point $x\in X$ is a sequence $p\in\N^\N$ such that $d(x,\alpha_{p(k)})<2^{-k}$ for any $k\in\N$.
This notion induces a partial surjection $\delta:\subseteq\N^\N\to X$ defined by
\[\delta(p)=x\iff p\mbox{ is a Cauchy name of }x.\]
This surjection $\delta$ is called the {\em Cauchy representation of }\index{Cauchy representation}$X$ (induced from $(d,\alpha)$).
One can show that $\delta$ is an admissible representation of $\mathcal{X}$.
\end{example}

\begin{fact}[de Brecht {\cite[Theorem 49]{deB13}}]\label{fact:total-representation}
Every Polish space has a total admissible representation.
Indeed, every quasi-Polish space has a total admissible representation.
\end{fact}

For represented spaces $\mathcal{X}=(X,\delta_X)$ and $\mathcal{Y}=(Y,\delta_Y)$, we say that a function $f\colon\mathcal{X}\to\mathcal{Y}$ is {\em symbolically continuous} if there exists a continuous function $F\pcolon Z\to Z$ such that, given a $\delta_X$-name ${\tt x}$ of $x\in X$, $F({\tt x})$ returns a $\delta_Y$-name of $f(x)$.
In other words, the following diagram commutes:
\[
\xymatrix{
X \ar[rr]^f & & Y \\
Z \ar[u]^{\delta_X} \ar[rr]_{F} & & Z\ar[u]_{\delta_Y}
}
\]

Such a function $F$ is called a {\em realizer} of $f$.
As long as we are dealing with admissible representations of Polish spaces (more generally, second countable $T_0$-spaces), there is no need to distinguish between symbolic continuity and topological continuity at all.

\begin{fact}[Schr\"oder {\cite[Theorem 4]{Sch02}}]
If $\mathcal{X}$ and $\mathcal{Y}$ are admissible represented spaces.
Then, $f$ is sequentially continuous if and only if $f$ is symbolically continuous.
\end{fact}

Let $\mathcal{X}=(X,\delta_X)$ be a represented space, and $\Gamma$ be a pointclass.
We say that $A\subseteq X$ is {\em symbolically $\Gamma$} if the set $\delta_X^{-1}[A]$ of all names of elements in $A$ is $\Gamma$ in ${\rm dom}(\delta)$.
As above, the symbolic Borel hierarchy and the topological Borel hierarchy coincide, and moreover, the symbolic difference hierarchy and the topological difference hierarchy coincide.
More precisely, let $D_\beta(\tpbf{\Sigma}^0_\alpha)$ be the $\beta$-th level of the difference hierarchy starting from $\tpbf{\Sigma}^0_\alpha$ sets, and then we have the following:

\begin{fact}[de Brecht {\cite[Theorem 68]{deB13}}]\label{fact:deBrecht-symbolic-complexity}
Let $\mathcal{X}$ be an admissibly represented second-countable $T_0$-space.
Then, a set $A\subseteq\mathcal{X}$ is $D_\beta(\tpbf{\Sigma}^0_\alpha)$ if and only if $A$ is symbolically $D_\beta(\tpbf{\Sigma}^0_\alpha)$.
\end{fact}

As an important observation, these results show, in particular, that the notion of symbolic complexity does not depend on the choice of admissible representation.
Callard-Hoyrup \cite[Theorem 4.1]{CaHo20} has shown the effective version of Fact \ref{fact:deBrecht-symbolic-complexity}.

We are interested in whether these results hold for the term classes $\tpbf{\Sigma}_t$.
Let $(X,\delta_X)$ be an admissibly represented space.
%Given a Polish space $X$, fix a total admissible representation $\delta_X$ of $X$.
The following result is due to Selivanov \cite{Sel19b}.

\begin{theorem}[Selivanov {\cite[Theorem 4.6]{Sel19b}}]\label{thm:Vaught-transform-main}
Let $\delta$ be an open admissible representation of a second-countable $T_0$ space $X$ with the domain $\|X\|\subseteq\om^\om$, and let $t$ be a normal $\lang(Q)$-term.
For a function $f\colon X\to Q$, $f$ is a $\tpbf{\Sigma}_t(X)$-function if and only if $f\circ\delta\pcolon Z\to Q$ is a $\tpbf{\Sigma}_t(\|X\|)$-function.
\end{theorem}

Let us say that a function $f\colon X\to Q$ is {\em symbolically $\tpbf{\Sigma}_t$} if $f\circ\delta_X\colon Z\to Q$ is $\tpbf{\Sigma}_t$, where $\delta_X$ is an admissible representation of $X$.
Then, Theorem \ref{thm:Vaught-transform-main} can be rephrased as follows:
\[\mbox{$f$ is $\tpbf{\Sigma}_t$}\iff\mbox{$f$ is symbolically $\tpbf{\Sigma}_t$}.\]

Hence, the symbolic complexity (with respect to flowcharts) of a Borel function is always the same as its topological complexity.
Moreover, this equivalence holds effectively.
However, our definition of $\tpbf{\Sigma}_t$ is slightly different from the definition adopted in Selivanov \cite{Sel19b}, and we would also like to add one more remark about the conclusion that follows from this theorem (see Section \ref{sec:symbolic-Wadge}).
Therefore, for the sake of completeness, in Section \ref{sec:Vaught-transform}, we will give a complete proof of Theorem \ref{thm:Vaught-transform-main}, which also shows that a flowchart is a useful notion for making the proof crystal clear.
This is also useful for clarifying that the proof of our main theorem in Section \ref{sec:main} has an analogous structure to that of Selivanov's Theorem \ref{thm:Vaught-transform-main}, namely the transformation of the sets assigned to a flowchart.

\subsection{Symbolic Wadge reducibility}\label{sec:symbolic-Wadge}
One may apply this result to the Wadge theory on admissibly represented spaces introduced by Pequignot \cite{Peq15}.
In this section, we show that, even in higher-dimensional Polish spaces, our definition of $\tpbf{\Sigma}_t$ by means of a flowchart can be thought of as a Wadge class.
%To explain this, let $(X,\delta_X)$ and $(Y,\delta_Y)$ be admissibly represented spaces.
%A {\em realizer} of a function $f\colon X\to Y$ is a function $F\pcolon\om^\om\to\om^\om$ such that if $p\in\om^\om$ is a name of $x\in X$ then $F(p)$ is a name of $f(x)$; that is, $f\circ\delta_X=\delta_Y\circ F$.
%Whenever $X$ and $Y$ are sequential space, it is known that $f\colon X\to Y$ is continuous if and only if $f$ has a continuous realizer.
%Hereafter we only consider second-countable $T_0$ spaces, and note that such spaces are sequential.
%We say that a multi-valued function $f\colon X\tto Y$ is {\em continuous} if $f$ has a continuous realizer $F$; that is, 
%Pequignot \cite{aaaa} introduced the notion of Wadge reducibility on subsets of admissibly represented spaces.

In contrast to the case of zero-dimensional spaces, it is known that the Wadge degrees in higher-dimensional spaces behaves badly \cite{ITS16,Schl17}.
For this reason, it is difficult to say that the Wadge degrees in higher-dimensional spaces is a useful measure of topological complexity.
To overcome this difficulty, Pequignot \cite{Peq15} introduced a modified version of Wadge reducibility.
Let $(X,\delta_X)$ and $(Y,\delta_Y)$ be admissibly represented spaces.

\begin{definition}[Pequignot \cite{Peq15}]\label{def:Pequignot-1}
We say that $A\subseteq X$ is {\em symbolic Wadge reducible} to $B\subseteq Y$ (written $A/X\sqleq_W B/Y$) if there exists a continuous function $\theta\pcolon\om^\om\to\om^\om$ such that for any $\delta_X$-name $p$,
\[\mbox{$p$ is a $\delta_X$-name of an element of $A$}\iff \mbox{$\theta(p)$ is a $\delta_Y$-name of an element of $B$}.\]

In other words, $A/X\sqleq_WB/Y$ states that $\delta_X^{-1}[A]$ is Wadge reducible to $\delta_Y^{-1}[B]$, but it is sufficient if such a reduction is defined only on the domain of $\delta_X$.
\end{definition}
This notion has also been studied in Camerlo \cite{Cam19}.
In general, we consider the following notion:

\begin{definition}\label{def:Pequignot-2}
Let $Q$ be a quasi-ordered set.
We say that $f\colon X\to Q$ is {\em symbolic Wadge reducible} to $g\colon Y\to Q$ (written $f\sqleq_Wg$) if there exists a continuous function $\theta\pcolon\om^\om\to\om^\om$ such that, whenever $p$ is a $\delta_X$-name of an element $x\in X$, $\theta(p)$ is a $\delta_Y$-name of an element $y\in Y$, and 
\[f(x)\leq_Q g(y).\]

In other words, $f\sqleq_Wg$ states that $f\circ\delta_X$ is Wadge reducible to $g\circ\delta_Y$.
\end{definition}

If $Q=\{0,1\}$ is equipped with the discrete order, then Definitions \ref{def:Pequignot-1} and \ref{def:Pequignot-2} for $Q$ coincide.
Diagrammatically, the definition of $f\sqleq_Wg$ may be described as follows:
% if and only if there exists a continuous function $\theta$ such that the following diagram commutes:
\[
\xymatrix{
\Baire \ar[rr]^{\delta_X} \ar[d]_\theta & & X \ar[rr]^{f} & & Q \ar[d]^{\leq_Q} \\
\Baire \ar[rr]_{\delta_Y} & & Y \ar[rr]_{g} & & Q
}
\]

It is easy to see that the definition of $\sqleq_W$ is independent of the choice of the admissible representations.
Hence, if a given space $Z$ is quasi-Polish, as in de Brecht \cite[Theorem 68]{deB13}, one can always assume that a representation of $Z$ is an open function and has Polish fibers.

Recall from Facts \ref{fact:KM-main} and \ref{fact:KM-main2}, the Wadge classes of Borel functions $\om^\om\to Q$ coincide with the classes $\tpbf{\Sigma}^{W}_t(\Baire)$ for $\mathcal{L}(Q)$-terms $t$.
Recall also that $\treeleq$ is the nested homomorphic quasi-order on the normal well-formed $\lang(Q)$-terms introduced by Kihara-Montalb\'an \cite{KM19}.
By Theorems \ref{prop:matryo-flow} and \ref{thm:Vaught-transform-main}, we obtain the following:

\begin{cor}\label{cor:symbolic-Wadge}
Let $X$ and $Y$ be Polish spaces, $Q$ be a better-quasi-ordered set, and $t$ be a well-formed $\lang(Q)$-term.
Then, for any function $f\colon X\to Q$, if $g\colon Y\to Q$ is $\tpbf{\Sigma}_t(Y)$,
\[f\sqleq_W g\iff f\in\tpbf{\Sigma}_s(X)\mbox{ for some $s\treeleq t$}.\]
\end{cor}

\begin{proof}
Every well-formed $\lang(Q)$-term $t$ is equivalent to a normal well-formed $\lang(Q)$-term with respect to $\treeleq$, see \cite{KM19}; hence, one can assume that $t$ is normal.
Let $\delta_X$ and $\delta_Y$ be total admissible representations of $X$ and $Y$, respectively.
By definition, $f\sqleq_{W}g$ if and only if $f\circ\delta_X\leq_Wg\circ\delta_Y$.
By Selivanov's Theorem \ref{thm:Vaught-transform-main}, $g\in\tpbf{\Sigma}_t(Y)$ if and only if $g\circ\delta\in\tpbf{\Sigma}_t(\Baire)$.
By Theorem \ref{prop:matryo-flow}, the latter condition is equivalent to $g\circ\delta\in\tpbf{\Sigma}_t^W(\Baire)$.
Hence, by Kihara-Montalb\'an's Theorem (Fact \ref{fact:KM-main2}),
\[f\circ\delta_X\leq_W g\circ\delta_Y\iff f\circ\delta_X\in\tpbf{\Sigma}^W_s(\Baire)\mbox{ for some $s\treeleq t$}.\]

Again by Theorem \ref{prop:matryo-flow}, $f\circ\delta_X\in\tpbf{\Sigma}^W_s(\Baire)$ if and only if $f\circ\delta_X\in\tpbf{\Sigma}_s(\Baire)$.
Hence, by Selivanov's Theorem \ref{thm:Vaught-transform-main}, the latter condition is equivalent to $f\in\tpbf{\Sigma}_s(X)$.
This concludes the proof.
\end{proof}

Let us describe a conclusion derived from Corollary \ref{cor:symbolic-Wadge}.
First, without knowing Corollary \ref{cor:symbolic-Wadge}, it is easy to see that the subsets of Polish spaces is semi-well-ordered under symbolic Wadge reducibility; see also Pequignot \cite{Peq15}.
Therefore, one can assign an ordinal rank to each subset of a Polish space $X$.
However, it is not immediately obvious what a pointclass in $X$ corresponding to a given ordinal rank $\alpha$ is.
Now, Corollary \ref{cor:symbolic-Wadge} makes this clear:
The pointclass in $X$ corresponding to the ordinal rank $\alpha$ is exactly what we expect it to be.

\subsection{Vaught transform of a flowchart}\label{sec:Vaught-transform}

Let $\delta$ be an admissible representation for a second-countable $T_0$-space $X$.
Then, for $x\in X$ we use $\|x\|$ to denote the set $\delta^{-1}\{x\}$ of all $\delta$-names of $x$.
As in de Brecht \cite[Theorem 68]{deB13}, one can assume that $\delta$ is an open function and has Polish fibers.
For a set $A\subseteq\om^\om$, define the {\em $\delta$-Vaught transform} $\delta^\ast[A]$ of $A$ (cf.~Saint-Raymond \cite[Lemma 17]{SR07}, de Brecht \cite[Theorem 68]{deB13} and Callard-Hoyrup \cite[Theorem 4.1]{CaHo20}) as follows:
\[\delta^\ast[A]=\{x\in X:\|x\|\cap A\mbox{ is not meager in }\|x\|\}.\]

\begin{fact}[Hyperarithmetical complexity of the forcing relation; cf.~\cite{SR07,deB13,CaHo20}]\label{dBCH}
Let $X$ be a second countable $T_0$-space with an admissible representation $\delta$.
Then, for any set $A\subseteq X$,
\[A\mbox{ is $\tpbf{\Sigma}^0_\alpha$}\implies \delta^\ast[A] \mbox{ is $\tpbf{\Sigma}^0_\alpha$}.\]

Indeed, given a $\tpbf{\Sigma}^0_\alpha$-code of $A$, one can effectively find a $\tpbf{\Sigma}^0_\alpha$-code of $\delta^\ast[A]$.
\end{fact}

Let $\bfS$ be a normal flowchart on an $\lang(Q)$-term $t$ over $\Baire$.
Then, we define the {\em $\delta$-Vaught transform of $\bfS$} as a flowchart $\bfS^\delta=(S^\delta_\sigma)_{\sigma\in\syn_t}$ on $t$ over $X$ as follows:
\begin{enumerate}
\item For a leaf $\sigma\in\syn_t$, $S^\delta_\sigma=X$.
\item If $\sigma\in\syn_t$ is labeled by $\mul$, then define $S_\sigma^\delta=\delta^\ast[S_\sigma]$.
\item If $\sigma\in\syn_t$ is labeled by $\ad$, then define $S_{\sigma,n}^\delta=\delta^\ast[S_{\sigma,n}]$ for each $n\in\om$.
\item If $\sigma\in\syn_t$ is labeled by $\phi_\alpha$, then $S^\delta_\sigma=X$.
\end{enumerate}

By Fact \ref{dBCH}, $\bfS^\delta$ is also a flowchart on $t$.
%Next, it is easy to check that $\delta^\ast[A\cap B]=\delta^\ast[A]\cap\delta^\ast[B]$ (since all fibers of $\delta$ are Polish, so Baire).
%We have the following property for the $\delta$-Vaught transform of the difference of sets.
By Lemma \ref{lem:monotone}, one can assume that $\bfS$ is monotone.

\begin{prop}[Selivanov {\cite[Lemma 4.5]{Sel19b}}]\label{thm:vaught-transform}
Let $X$ be a second-countable $T_0$ space with an admissible representation $\delta$, and $f\colon X\to Q$ be a function.
If $\bfS$ is a monotone flowchart determining $f\circ\delta$, then $\bfS^\delta$ is a flowchart determining $f$.
\end{prop}

\begin{proof}
As above, one can assume that $\delta$ has Polish fibers, that is, $\|x\|$ is Polish for any $x\in X$.
Let $(D_\sigma)_{\sigma\in\syn_t}$ and $(E_\sigma)_{\sigma\in\syn_t}$ be the domain assignments to $\bfS$ and $\bfS^\delta$, respectively.
We first show the following claim:

\begin{claim}\label{claim:vaught-transform}
$E_\sigma\subseteq \delta^\ast[D_\sigma]$.
\end{claim}

\begin{proof}
We first observe that $\delta^\ast[A]\setminus\delta^\ast[B]\subseteq\delta^\ast[A\setminus B]$.
To see this, assume that $x\in\delta^\ast[A]\setminus\delta^\ast[B]$.
Then $x\in\delta^\ast[A]$ means that $\|x\|\cap A$ is not meager in $\|x\|$, and $x\not\in\delta^\ast[B]$ means that $\|x\|\cap B$ is meager, so $\|x\|\setminus B$ is comeager, in $\|x\|$.
Hence, by the Baire category theorem on the Polish fiber $\|x\|$, we see that $\|x\|\cap(A\setminus B)$ is not meager in $\|x\|$.
Therefore, we have $x\in\delta^\ast[A\setminus B]$.

%Given a flowchart $\bfS$ on $t$, the domain assignment to $\bfS$ is defined by only using intersection and set difference.
%By Lemma \ref{lem:monotone}, one can assume that $\bfS$ is monotone.
%Then, clearly, $\bfS^\delta$ is also monotone.

By induction on the syntax tree $\syn_t$.
If $\sigma$ is labeled by $\mul$, then by induction hypothesis and the above observation on the difference of sets,
\[E_{\sigma\fr 0}=E_\sigma\setminus\delta^\ast[S_\sigma]\subseteq\delta^\ast[D_\sigma]\setminus\delta^\ast[S_\sigma]\subseteq\delta^\ast[D_\sigma\setminus S_\sigma]=\delta^\ast[D_{\sigma\fr 0}].\]

By monotonicity of $\bfS$, we have $D_{\sigma\fr 1}=S_\sigma$.
Hence, by induction hypothesis, 
\[E_{\sigma\fr 1}=E_\sigma\cap\delta^\ast[S_\sigma]\subseteq\delta^\ast[D_\sigma]\cap\delta^\ast[S_\sigma]=\delta^\ast[D_\sigma]\cap\delta^\ast[D_{\sigma\fr 1}]=\delta^\ast[D_{\sigma\fr 1}]\]
since $D_{\sigma\fr 1}\subseteq D_\sigma$.
The same argument applies to the case that $\sigma$ is labeled by $\ad$.
\end{proof}

Next, note that $\bfS$ must be a total flowchart over $X$ since $f$ is total.
Again, let $(D_\sigma)_{\sigma\in\syn_t}$ and $(E_\sigma)_{\sigma\in\syn_t}$ be the domain assignments to $\bfS$ and $\bfS^\delta$, respectively.
To show that $\bfS^\delta$ is also total, it suffices to show that $E_\sigma\subseteq\bigcup_{n\in\om}\delta^\ast[S_{\sigma,n}]$ whenever $\sigma$ is labeled by $\ad$.
By Claim \ref{claim:vaught-transform}, $x\in E_\sigma$ implies $x\in\delta^\ast[D_\sigma]$.
Hence, $\|x\|\cap D_\sigma$ is not meager in $\|x\|$.
By totality of $\bfS$, we have $D_\sigma\subseteq\bigcup_{n\in\om}S_{\sigma,n}$.
Therefore, by the Baire category theorem on $\|x\|$, there exists $n\in\om$ such that $\|x\|\cap S_{\sigma,n}$ is not meager in $\|x\|$.
This means that $x\in\delta^\ast[S_{\sigma,n}]$.
This concludes that $\bfS^\delta$ is total.

%Then, for each leaf $\rho\in\syn_t$, $D_\rho$ is the set of all $z$ such that $\rho$ is a true path for $z$ with respect to $\bfS$.
%Note that $\bfS$ must be a total flowchart over $X$, and the syntax tree $\syn_t$ is countable,
%Therefore, by the Baire category theorem, there exists a leaf $\rho\in\syn_t$ such that $\|x\|\cap D_\rho$ is not meager in $\|x\|$.
%This means that $x\in\delta^\ast[D_\rho]$.
%?????

Let $\rho$ be a true path for $x$ with respect to $\bf{S}^\delta$; that is, $x\in E_\rho$.
If the leaf $\rho$ is labeled by $q_\rho$, then we have $\eval{\bfS^\delta}(x)=q_\rho$.
By the above claim, we also have $x\in\delta^\ast[D_\rho]$.
In particular, there is $p\in\delta^{-1}\{x\}$ such that $p\in D_\rho$.
Thus, $\eval{\bfS}(p)=q_\rho$.
By our assumption, $\bfS$ is a flowchart determining $f\circ\delta$, and therefore, $f\circ\delta(p)=\eval{\bfS}(p)$.
Hence,
\[\eval{\bfS^\delta}(x)=q_\rho=\eval{\bfS}(p)=f\circ\delta(p)=f(x).\]

This shows that $\bfS^\delta$ is a flowchart determining $f$.
\end{proof}

\begin{proof}[Proof of Theorem \ref{thm:Vaught-transform-main}]
It is not hard to check that the class $\tpbf{\Sigma}_t$ is closed under continuous substitution; that is, if $f\in\tpbf{\Sigma}_t(X)$, then $f\circ\delta\in\tpbf{\Sigma}_t(\|X\|)$.
For the converse direction, assume $f\circ\delta\in\tpbf{\Sigma}_t(\|X\|)$.
Then we have a flowchart $\bfS$ on $t$ determining $f\circ\delta$.
By Proposition \ref{thm:vaught-transform}, its $\delta$-Vaught transform $\bfS^\delta$ is a flowchart on $t$ determining $f$.
Consequently, $f$ is a $\tpbf{\Sigma}_t$-function.
\end{proof}

\section{Louveau-type effectivization}\label{sec:main}

\subsection{Main results}

In this section, we show that, if a Borel function between Polish spaces happens to be a $\tpbf{\Sigma}_t$ function, then its $\tpbf{\Sigma}_t$-code can be obtained from its Borel code in a hyperarithmetical manner.

\begin{theorem}\label{thm:Louveau-for-functions}
Let $X$ be a computable Polish space, $Q$ be a computable Polish space, and $t$ be a hyperarithmetical $\lang(Q)$-term.
Then, $\tpbf{\Sigma}_t(X)\cap\Delta^1_1=\Sigma_t(\Delta^1_1;X)$.
\end{theorem}

This can be viewed as a generalization of Louveau's theorem \cite[Theorem A]{Lou80}, which states that, if $\tpbf{\Gamma}$ is a Wadge class of $\Baire$ which has some $\Delta^1_1$-description, then $\Gamma(\Delta^1_1)=\tpbf{\Gamma}\cap \Delta^1_1$.
Essentially this is a special case of Theorem \ref{thm:Louveau-for-functions} with $Q=\{0,1\}$, although our term-description is different from Louveau's one.
One can also show that Theorem \ref{thm:Louveau-for-functions} holds for computable quasi-Polish spaces.
Here, roughly speaking, a computable quasi-Polish space is a represented second countable $T_0$ space which is the image of a computable open surjection from $\Baire$ (which almost corresponds to the effective version of Fact \ref{fact:total-representation}; see also \cite{BPS,HRSS}).

To further generalize Louveau's theorem, let us note that separating pairwise disjoint sets is a special case of extending a partial function.
To see this, let us identify a subset $S$ of $X$ with its characteristic function $\chi_S\colon X\to\{0,1\}$.
Then, for example, the Lusin's separation theorem for $\tpbf{\Sigma}^1_1$ sets is interpreted as the statement saying that there exists a total $2$-valued $\tpbf{\Delta}^1_1$-function which extends a given partial $2$-valued $\tpbf{\Sigma}^1_1$-function. 

Note that if the codomain $Q$ is countable, $\tpbf{\Sigma}^1_1$-functions and $\tpbf{\Pi}^1_1$-functions are in fact the same class since $f(x)=q$  if and only if, for any $p\in Q$, $p\not=q$ implies $f(x)\not=q$.
Hence, instead of $\tpbf{\Sigma}^1_1$-functions, one may consider $\tpbf{\Pi}^1_1$-functions with $\tpbf{\Sigma}^1_1$-domains.
We show the following generalization of Louveau's separation theorem (see Fact \ref{fact:Louveau-sep}):

%\begin{definition}
%Let $\mathcal{X}$ be an effective Polish space.
%Let $A: \subseteq\mathcal{X} \to Q$ be a partial $Q$-valued function. We say that $B:\mathcal{X} \to Q$ extends $A$ if $A(x)=B(x)$ for all $x\in dom(A)$.
%\end{definition}

\begin{theorem}[Effective extension]\label{thm:main-theorem-extension}
Let $Q$ be a computable Polish space, $t(\bar{x})$ be a hyperarithmetical $\lang(Q)$-term, and $f\pcolon\Baire\to Q$ be a partial $\Pi^1_1$-measurable function with a $\Sigma^1_1$ domain.
Suppose that $f$ can be extended to a total $\tpbf{\Sigma}_{t}$ function $g\colon\Baire\to Q$.
Then, $f$ can be extended to a total ${\Sigma}_{t}(\Delta^1_1)$ function $g^\star\colon\Baire\to Q$.
\end{theorem}

Indeed, Theorem \ref{thm:main-theorem-extension} can be viewed as a functional version of Louveau's theorem \cite[Theorem 2.7]{Lou80} for Borel Wadge classes.
However, our language $\lang(Q)$ is designed as a tool to discuss the case where $Q$ is a quasi-ordered set.
In such a case, it is natural to deal with the domination theorem rather than the extension theorem.
Let $\leq_Q$ be a quasi-order on $Q$.
For sets $A\subseteq B$ and functions $f\colon A\to Q$ and $g\colon B\to Q$, we say that {\em $f$ is $\leq_Q$-dominated by $g$} if $f(x)\leq_Qg(x)$ for any $x\in A$.

\begin{theorem}[Effective domination]\label{thm:main-theorem-domination}
Let $\leq_Q$ be a $\Pi^1_1$ quasi-order on a computable Polish space $Q$, $t(\bar{x})$ be a hyperarithmetical $\lang(Q)$-term, and $f\pcolon\om^\om\to Q$ be a partial $\Pi^1_1$-measurable function with a $\Sigma^1_1$ domain.
Suppose that $f$ is $\leq_Q$-dominated by some total $\tpbf{\Sigma}_{t}$ function $g\colon\om^\om\to Q$.
Then, $f$ is $\leq_Q$-dominated by some total ${\Sigma}_{t}(\Delta^1_1)$ function $g^\star\colon\om^\om\to Q$.
\end{theorem}

Note that $g$ extends $f$ if and only if $f$ is $=$-dominated by $g$.
Thus, Theorem \ref{thm:main-theorem-domination} is more general than Theorem \ref{thm:main-theorem-extension}.
Effective Domination Theorem \ref{thm:main-theorem-domination} can be applied, for example, to ordinal-valued functions; see also Example \ref{ex:ordinal-valued-func}.
For studies on ordinal-valued functions in the context of Wadge degrees, see also \cite{Dup03,Blo14}.

As another different type of application of our main techniques, we next consider decomposability of Borel functions.
The study of decomposability of Borel functions originates from Lusin's old question asking whether any Borel function can be written as a union of countably many partial continuous functions.
In the study of decomposability of Borel function, it is known that computability-theoretic analysis plays an important role, as shown in Kihara \cite{Kih15} and Gregoriades-Kihara-Ng \cite{GKN21}.
The decomposition of a Borel function into partial continuous functions along a flowchart has been studied in \cite{Kih16}.
One can express such a decomposition as an $\lang(\mathcal{C})$-term where each partial continuous function is considered as a constant symbol, i.e., $\mathcal{C}$ is the set of partial continuous functions.

More precisely, for an $\lang(\mathcal{C})$-term, a function $f\colon X\to Y$ is in $\tpbf{\Sigma}_t$ if there exists a flowchart $\bfS$ on $t$ such that, for any $x\in X$ and true path $\rho$ for $x$ w.r.t.~$\bfS$, if $\rho$ is labeled by $h$, then $f(x)=h(x)$.
Note that $\bfS$ is not necessarily deterministic in the sense of Section \ref{sec:flowchart}, but still deterministic in a certain sense.
A $\tpbf{\Sigma}_t$-function is clearly Borel-piecewise continuous.
However, this definition is too restrictive as the constant symbols in the term $t$ already tell us what partial continuous functions it can be decomposed into.
To weaken this restriction, we give another definition of $\tpbf{\Sigma}_t$-piecewise continuity using an open term, which make the part of partial continuous functions undetermined.

Recall that variable symbols are indexed as $(x_j)_{j\in\Baire}$.
For an open $\lang$-term $t(\bar{x})$ and a set $Q$, a {\em $Q$-valuation of $t(\bar{x})$} is a sequence $\bar{a}=(a_j)_{j\in\om^\om}$ where $a_j\in Q$ for any $j\in\Baire$.
Then, by $t(\bar{a})$, we denote the result of substituting $a_j$ for $x_j$ in $t(\bar{x})$ for each $j\in\Baire$.
We say that a function $f\pcolon X\to Y$ is {\em $\tpbf{\Sigma}_{t(\bar{x})}$-piecewise continuous} if there exists a $\mathcal{C}$-valuation $\bar{a}$ such that $f$ is in $\tpbf{\Sigma}_{t(\bar{a})}$.
Then, $f$ is {\em $\tpbf{\Sigma}_{t(\bar{x})}$-piecewise $\Delta^1_1$-continuous} if such an $a_j$ is a $\Delta^1_1$-continuous function with a $\Delta^1_1$-domain for any $j\in\Baire$.
We say that $f$ is {\em ${\Sigma}_{t(\bar{x})}(\Delta^1_1)$-piecewise continuous} if such a $\mathcal{C}$-valuation $\bar{a}\colon\Baire\to\mathcal{C}$ is $\Delta^1_1$-measurable and $f$ is in $\Sigma_{t(\bar{a})}(\Delta^1_1)$, where $\mathcal{C}$ is considered as a represented space as usual (see \cite{CCA}).

\begin{theorem}[Effective decomposition]\label{thm:decomposition}
Let $X$ be a computable Polish space and $t(\bar{x})$ be a hyperarithmetical $\lang$-term.
If $f\colon X\to Y$ is $\Delta^1_1$-measurable and $\tpbf{\Sigma}_{t(\bar{x})}$-piecewise $\Delta^1_1$-continuous, then it is ${\Sigma}_{t(\bar{x})}(\Delta^1_1)$-piecewise continuous.
\end{theorem}

\subsection{Auxiliary tools}
As a common setting for dealing with our main results, we introduce a few auxiliary notions:
For sets $X$, $Y$ and $Z$, and a binary relation $\lhd\subseteq Y\times Z$, we say that a partial function $f\pcolon X\to Y$ is {\em $\lhd$-dominated by} $g\colon X\to Z$ if $f(x)\lhd g(x)$ whenever $x\in{\rm dom}(f)$.
Let $X$ be a computable Polish space, and $Q$ be a totally represented set, i.e., $Q=\{q_z\}_{z\in\Baire}$.
For a pointclass ${\Gamma}$, we say that $f$ is {\em ${\Gamma}$-measurable w.r.t. $\lhd$} if there exists a ${\Gamma}$ set $R\subseteq X\times\Baire$ such that for any $x\in X$,
\[x\in{\rm dom}(f)\implies(f(x)\lhd q_z\iff (x,z)\in R).\]

\begin{example}\label{exa:equality-complexity}
For Polish spaces $X$ and $Y$, if $f\colon X\to Y$ is a partial $\tpbf{\Sigma}^1_1$-measurable function, then $f$ is $\tpbf{\Pi}^1_1$-measurable w.r.t.~the equality.
To see this, note that $f(x)=y$ if and only if $x\in f^{-1}\{y\}$, which is $\tpbf{\Pi}^1_1$ since any singleton is closed.

If $f\colon X\to Y$ is a partial $\tpbf{\Pi}^1_1$-measurable function, then $f$ is $\tpbf{\Delta}^1_1$-measurable w.r.t.~the equality.
To see this, let $(B_e)_{e\in\om}$ be an enumeration of rational open balls in $X$.
Then, whenever $x\in{\rm dom}(f)$, we have the following:
\begin{align*}
f(x)=y&\iff(\forall e\in\om)\;[y\in B_e\implies f(x)\in B_e],\\
f(x)\not=y&\iff(\exists d,e\in\om)\;[B_d\cap B_e=\emptyset\;\land\;y\in B_d\;\land\;f(x)\in B_e].
\end{align*}

It is easy to see that the both formulas are $\tpbf{\Pi}^1_1$ since $f$ is $\tpbf{\Pi}^1_1$-measurable and $B_e$ is open.
\end{example}

\begin{example}\label{exa:order-complexity}
Let $X$ and $Q$ be Polish spaces.
Moreover, let $f\pcolon X\to Q$ be a partial $\tpbf{\Pi}^1_1$-measurable function, and $\lhd$ be a $\tpbf{\Pi}^1_1$ binary relation on $Q$.
Then, $f$ is $\tpbf{\Pi}^1_1$-measurable w.r.t.~$\lhd$, since whenever $x\in{\rm dom}(f)$ we have
\[f(x)\lhd z\iff(\forall q\in Q)\;[f(x)=q\implies q\lhd z].\]
%
%If $f$ is partial $\Pi^1_1$-measurable, and $\lhd$ is $\Pi^1_1$, then $f$ is $\Pi^1_1$-measurable w.r.t.~$\lhd$.
\end{example}

\begin{example}\label{ex:ordinal-valued-func}
Let ${\sf WO}$ be the set of all well-orders on $\N$, and define a binary relation $\lhd$ on ${\sf WO}$ as follows:
\[\alpha\lhd\beta\iff\mbox{the order type of $\alpha$ is less than or equal to that of $\beta$.}\]

We think of ${\sf WO}$ as a subspace of the Polish space $2^{\N\times\N}$.
By the usual ordinal comparison argument (see \cite[Theorem 4A.2]{Mos09}), if $f\pcolon X\to{\sf WO}$ is a partial $\tpbf{\Pi}^1_1$-measurable function, then $f$ is $\tpbf{\Pi}^1_1$-measurable w.r.t.~$\lhd$.
\end{example}

As a tool for proving our main result, we use the following topology, which is known to be useful in transforming boldface arguments into lightface ones.

\begin{definition}
For a computable Polish space $X$, the {\it Gandy-Harrington topology} $T_\infty$ on $X$ is the topology generated by $\Sigma^1_1$ subsets of $X$.
We use $T_\xi$ to denote the subtopology generated by sets in $\bigcup_{\eta<\xi}\tpbf{\Pi}^0_\eta \cap \Sigma^1_1$.
\end{definition}

For the basics of the Gandy-Harrington topology $T_\infty$ and its subtopology $T_\xi$, we refer the reader to Louveau \cite{Lou80}.

\begin{fact}[see {\cite[Proposition 6]{Lou80}}]\label{fact:GH-Baire}
The Gandy-Harrington topology $T_\infty$ is Baire in the topological sense:
Every nonempty $T_\infty$-open subset of $X$ is $T_\infty$-nonmeager.
\end{fact}

In the following, we write $\forall^\infty x P(x)$ to denote that $\{x:\lnot P(x)\}$ is $T_\infty$-meager.
In this case, we also say that {$P$ holds $\infty$-a.e.}
The following fact plays a key role in our proof.

\begin{fact}[Louveau {\cite[Lemma 8]{Lou80}}]\label{approximate}
Let $A$ be a $\tpbf{\Sigma}^0_\xi$ set in a computable Polish space $X$.
Then, there exists a $T_\xi$-open set $A^\prime$ such that $A=A^\prime$ $\infty$-a.e.
\end{fact}

Louveau used these facts to prove the so-called Louveau separation theorem:

\begin{fact}[Louveau {\cite[Theorem B]{Lou80}}]\label{fact:Louveau-sep}
If a disjoint pair of $\Sigma^1_1$ sets is separated by a $\tpbf{\Sigma}^0_\xi$ set, then it is separated by a $\Sigma^0_\xi(\Delta^1_1)$ set.
\end{fact}

%\begin{theorem}[Louveau]\label{Louveausep}
%Let $A, B$ be disjoint $\Sigma^1_1$ sets. Let $\bfGamma=\bfGamma_u$ be a Wadge class of Borel sets where $u$ is $\Delta^1_1$-description.  Suppose there is some $C\in \tpbf{\Gamma}$ such that $A\subseteq C \subseteq \mathcal{N}\setminus B$ $\infty$-a.e. where $\tpbf{\Gamma}$ is a Borel Wadge class.
%Then, there is some $D\in \Gamma(\Delta^1_1)$ such that $A\subseteq D \subseteq \mathcal{N}\setminus B$.
%\end{theorem}
%
%\begin{definition}
%Let $A: \subseteq\mathcal{X} \to Q$ be a partial $Q$-valued function. We say that $B:\mathcal{X} \to Q$ $\infty$-a.e. extends $A$ if $A(x)=B(x)$ for all $x\in dom(A)$ modulo $T_{\infty}$-meager set, i.e., $\{x\in dom(A) \mid A(x)\neq B(x)\}$ is $T_{\infty}$-meager. 
%\end{definition}

%Now, we prove our main theorem in this section. In the proof of Theorem \ref{Louveausep},
%he picked the largest $T_\xi$-open set which separates given two sets $\infty$-a.e.  and showed that it really separates them. We borrow this idea for our proof.

We also use a few basic facts in classical computability theory.
First, the following is known as Kreisel's $\Delta^1_1$-selection theorem; see e.g.~{\cite[4B.5]{Mos09}} or \cite[Theorem II.2.3]{Sac90}.

\begin{fact}[$\Delta^1_1$-selection]\label{fact:selection}
Let $X$ be a computable Polish space, and $E\subseteq X\times\om$ be a total $\Pi^1_1$ relation, i.e., for any $x\in X$ there exists $n\in\om$ such that $E(x,n)$ holds.
Then there exists a $\Delta^1_1$ function $f\colon X\to\om$ such that $(x,f(x))\in E$ holds for any $x\in X$.
\end{fact}

A formula $\varphi(X)$ with a set variable is $\Pi^1_1$ on $\Pi^1_1$ if for each uniformly $\Pi^1_1$ sequence $(A_n)_{n\in\om}$, the set $\{n\in\om: \varphi(A_n)\}$ is $\Pi^1_1$.
The following is known as the first $\Pi^1_1$-reflection theorem (the lightface version of \cite[{Theorem 35.10}]{Kec95}):

\begin{fact}[$\Pi^1_1$-reflection]\label{fact:reflection}
Let $\varphi$ be $\Pi^1_1$ on $\Pi^1_1$, and $A\in \Pi^1_1$. If $\varphi(A)$, then
 there is some $B\in \Delta^1_1$ such that $B\subseteq A$ and $\varphi(B)$.
\end{fact}

\subsection{Proof of Main Theorems}

In this section we prove Theorems \ref{thm:Louveau-for-functions}, \ref{thm:main-theorem-extension}, \ref{thm:main-theorem-domination} and \ref{thm:decomposition} by induction on the complexity of $\lang(Q)$-terms.
However, in order for the induction to work, we have to prove the following stronger claim:

\begin{lemma}\label{thm:main-theorem}
Let $H$ be a $\Sigma^1_1$ subset of $\Baire$, $Q$ be a set, $t$ be a hyperarithmetical $\lang(Q)$-term, $\lhd\subseteq Y\times Q$ be a binary relation, and $f\pcolon H\to Y$ be a partial function which is $\Pi^1_1$-measurable w.r.t.~$\lhd$ on its $\Sigma^1_1$ domain.
Suppose that $f$ is $\infty$-a.e.~$\lhd$-dominated by some $\tpbf{\Sigma}_{t}$ function $g\colon H\to Q$, i.e.,
\[\{x\in X:x\in{\rm dom}(f)\;\land\;\neg (f(x)\lhd g(x))\}\mbox{ is $T_\infty$-meager.}\]
Then, $f$ is $\lhd$-dominated by some ${\Sigma}_{t}(\Delta^1_1)$ function $g^\star\colon H\to Q$.
\end{lemma}

By effectivizing the arguments in Examples \ref{exa:equality-complexity} and \ref{exa:order-complexity} to calculate the lightface complexity, we observe that Lemma \ref{thm:main-theorem} implies Theorems \ref{thm:main-theorem-extension} and \ref{thm:main-theorem-domination}, which also deduces Theorem \ref{thm:Louveau-for-functions} for $X=\om^\om$.
To prove Theorem \ref{thm:Louveau-for-functions} for computable Polish $X$, first note that if $X$ is computable Polish (indeed, if $X$ is computable quasi-Polish; see \cite{BPS,HRSS}), then there exists a total computable open surjection $\delta\colon\Baire\to X$; see also Fact \ref{fact:total-representation}.
Then for a function $f\colon X\to Q$, $f\in\tpbf{\Sigma}_t(X)\cap\Delta^1_1$ implies $f\circ\delta\in\tpbf{\Sigma}_t(\Baire)\cap\Delta^1_1$; hence by Lemma \ref{thm:main-theorem}, we have $f\circ\delta\in\Sigma_t(\Delta^1_1;\Baire)$, and thus $f\in\Sigma_t(\Delta^1_1;X)$ by Selivanov's Theorem \ref{thm:Vaught-transform-main}.

To prove Theorem \ref{thm:decomposition}, we show the multi-valued version of Lemma \ref{thm:main-theorem} for any computable Polish space $X$.
Recall from Section \ref{sec:flowchart} that it is straightforward to define $\tpbf{\Sigma}_t$ and $\Sigma_t(\Delta^1_1)$ even for multi-valued functions by considering possibly non-deterministic flowcharts.
If $g$ is multi-valued, then we say that $f$ is $\lhd$-dominated by $g$ if, for any $x\in{\rm dom}(f)$ and any value $z$ of $g(x)$, we have $f(x)\lhd z$.

\begin{lemma}\label{thm:main-theorem-multi}
Let $H$ be a $\Sigma^1_1$ subset of a computable Polish space $X$, $Q$ be a set, $t$ be a hyperarithmetical $\lang(Q)$-term, $\lhd\subseteq Y\times Q$ be a binary relation, and $f\pcolon H\to Y$ be a partial function which is $\Pi^1_1$-measurable w.r.t.~$\lhd$ on its $\Sigma^1_1$ domain.
Suppose that $f$ is $\infty$-a.e.~$\lhd$-dominated by some multi-valued $\tpbf{\Sigma}_{t}$ function $g\colon H\rightrightarrows Q$, i.e.,
\[\{x\in X:x\in{\rm dom}(f)\;\land\;\exists z\in g(x)\;\neg (f(x)\lhd z)\}\mbox{ is $T_\infty$-meager.}\]
Then, $f$ is $\lhd$-dominated by some multi-valued ${\Sigma}_{t}(\Delta^1_1)$ function $g^\star\colon H\rightrightarrows Q$.
\end{lemma}

%To prove Theorem \ref{thm:Louveau-for-functions}, let $f\in\tpbf{\Sigma}_t\cap\Delta^1_1$ be given.
%Then, $f$ is $\lhd$-dominated by $f$ itself, where $\lhd$ is the equality.
%Therefore, by Lemma \ref{thm:main-theorem-multi} with $H=X$, $f$ is $\lhd$-dominated by a multi-valued ${\Sigma}_{t}(\Delta^1_1)$ function $g^\star$.
%However, this means that, for any value $z$ of $g^\star(x)$, $f(x)=z$; that is, $g^\star$ is single-valued and equal to $f$.
%This concludes $f\in\Sigma_t(\Delta^1_1)$.

Assuming Lemma \ref{thm:main-theorem-multi}, let us give the detailed proof of Theorem \ref{thm:decomposition}.

\begin{proof}[Proof of Theorem \ref{thm:decomposition}]
%We first see that if $g$ is $(\tpbf{\Sigma}_{t(\bar{x})}\cap\Delta^1_1)$-piecewise continuous, then it is $(\tpbf{\Sigma}_{t(\bar{x})}\cap\Delta^1_1)$-piecewise $\Delta^1_1$-continuous.
%To see this, by our assumption, we have a hyperarithmetical $\tpbf{\Sigma}_{t(\bar{x})}$-function $j\colon X\to J$.
%Let $\eta$ be a $J$-flowchart on $t(\bar{x})$, and $(D_\sigma)_{\sigma\in\syn_t}$ be its domain assignment.
%Let $\rho$ be a leaf of the syntax tree $\syn_t$.
%Given an open set $U$, search for an index of $V$ such that $D_\sigma\cap g^{-1}[U]=D_\sigma\cap V$.
%This is a $\Pi^1_1$ property since
%\[(\forall x)\;[x\in D_\sigma\implies(x\in g^{-1}[U]\iff x\in V)].\]
Assume that $f$ is $\tpbf{\Sigma}_{t(\bar{x})}$-piecewise $\Delta^1_1$-continuous.
Then there exists a $\mathcal{C}(\Delta^1_1)$-valuation $\bar{h}=(h_j)_{j\in\Baire}$ such that $f$ is a $\tpbf{\Sigma}_{t(\bar{h})}$-function.
%Let $\bfS$ be a (possibly non-deterministic) flowchart on $t(\bar{h})$ determining $f$.
For each $e\in\om$, let $\varphi_e$ be the partial $\Pi^1_1$-continuous function from $X$ to $Y$ coded by $e$, and let $I$ be the $\Pi^1_1$ set of all codes $e$ such that $\varphi_e$ is indeed $\Delta^1_1$.
Then, there exists a function $d\colon\Baire\to I$ such that $h_j=\varphi_{d(j)}$ for any $j\in\Baire$.
Put $\bar{d}=(d_j)_{j\in\Baire}$ and then indeed, the definition of $f$ involves a (possibly non-deterministic) flowchart $[\bfS]$ on the $\lang(\om)$-term ${t(\bar{d})}$.
Clearly, $[\bfS]$ yields a multi-valued $\tpbf{\Sigma}_{t(\bar{d})}$-function $g$ such that, for any $x\in{\rm dom}(f)$ and any value $e$ of $g(x)$, we have $e\in I$ and $f(x)=\varphi_{e}(x)$.

Define $\lhd\subseteq(X\times Y)\times \om$ by $(x,y)\lhd e$ if and only if $e\in I$, $x\in{\rm dom}(\varphi_{e})$ and $\varphi_{e}(x)=y$.
Let us consider the function $({\rm id},f)\colon X\to X\times Y$ defined by $({\rm id},f)(x)=(x,f(x))$.
Since $f$ is $\Delta^1_1$-measurable, and each $\varphi_e$ is a $\Delta^1_1$-continuous function with a $\Delta^1_1$-domain, the conditions $e\in I$, $x\in {\rm dom}(\varphi_{e})$ and $f(x)=\varphi_{e}(x)$ are $\Pi^1_1$; hence, $({\rm id},f)$ is $\Pi^1_1$-measurable w.r.t.~$\lhd$.
For a single-valued function $h$, $({\rm id},f)$ is $\lhd$-dominated by $h$ if and only if $(x,f(x))\lhd h(x)$ for any $x\in{\rm dom}(f)$ if and only if $h(x)\in I$ and $f(x)=\varphi_{h(x)}(x)$ for any $x\in{\rm dom}(f)$.
Therefore, for a multi-valued function $h$, $({\rm id},f)$ is $\lhd$-dominated by $h$ if and only if, for any $x\in{\rm dom}(f)$ and $e\in h(x)$, we have $e\in I$ and $f(x)=\varphi_e(x)$.
Therefore, $({\rm id},f)$ is $\lhd$-dominated by $g$.
Now, by Lemma \ref{thm:main-theorem-multi}, $({\rm id},f)$ is $\lhd$-dominated by a multi-valued $\Sigma_{t(\bar{d})}(\Delta^1_1)$-function $g^\star$.
Then, we have $f(x)=\varphi_e(x)$ for any $e\in g^\star(x)$ as above.
Let $\bfS'$ be a flowchart determining $g^\star$, and then define $\bfS''$ by replacing the label $e\in\om$ of each leaf $\rho\in\syn_{t(\bar{d})}$ with $\varphi_e$.
It is easy to see that this flowchart $\bfS''$ witnesses that $f$ is ${\Sigma}_{t(\bar{x})}(\Delta^1_1)$-piecewise continuous.
\end{proof}

Before starting the proof of Lemma \ref{thm:main-theorem}, let us look at the syntax tree of an $\lang(Q)$-term $t$.
Then we can immediately see that the syntax tree $\syn_t$ has the uppermost node $\sigma'$ (i.e., the node closest to the root) which is not labeled by a Veblen function symbol.
%Note that $\sigma'$ is also the first branching node in $\syn_t$.
Let us call $\sigma'$ the {\em neck node of $\syn_t$}.
In other words, the term $t$ can always be written as 
\[t=t'\quad\mbox{ or }\quad t=\phi_{\alpha_0}(\phi_{\alpha_{1}}(\dots(\phi_{\alpha_\ell}(t')))),\]
where the root of the syntax tree $\syn_{t'}$ of the subterm $t'$ is not labeled by a Veblen function symbol.
In this case, we write $\phi_{\alpha_0}(\phi_{\alpha_{1}}(\dots(\phi_{\alpha_\ell}(\cdot))))$ as $\phi_t(\cdot)$, and let us say that $\phi_t$ is the head of $t$, and $t'$ is the body of $t$.
Then, $t$ is of the form $\phi_t(t')$.
See also \cite[Lemma 3.5]{Sel19b}.
%Then, we define the Borel rank of the body of $t$ as $\rank(\sigma')$, where $\sigma'$ is the neck node of $\syn_t$.
It is straightforward to see the following:

\begin{obs}\label{obs:Veblen-head-body}
Let $f\colon X\to Q$ be a function, and $\xi$ be the Borel rank of the neck node of an $\lang(Q)$-term $t$.
\begin{enumerate}
\item In the case where $t$ is of the form $\phi_t(q)$ for some $q\in Q$, $f\in\tpbf{\Sigma}_{t}$ if and only if $f\in\tpbf{\Sigma}_q$.
\item In the case where $t$ is of the form $\phi_t(s\mul u)$, $f\in\tpbf{\Sigma}_t$ if and only if there exists a $\tpbf{\Sigma}^0_\xi$ set $U\subseteq X$ such that $f\upto(X\setminus U)\in\tpbf{\Sigma}_{\phi_t(s)}$ and $f\upto U\in\tpbf{\Sigma}_{\phi_t(u)}$.
\item In the case where $t$ is of the form $\phi_t(\sqcup_{i\in\om}s_i)$, $f\in\tpbf{\Sigma}_{t}$ if and only if there exists a $\tpbf{\Sigma}^0_\xi$ cover $(U_n)_{n\in\om}$ of $X$ such that $f\upto U_n\in\tpbf{\Sigma}_{\phi_t(s_n)}$ for any $n\in\om$.
\end{enumerate}
\end{obs}

Now, let us start the proof of Main Lemma \ref{thm:main-theorem}.

\begin{proof}[Proof of Lemma \ref{thm:main-theorem}]
We prove the assertion by induction on the complexity of the body of $t=\phi_t(t')$, where $t'$ is the body of $t$.
Inductively assume that we have already shown the assertion for the term $\phi_t(t'')$ for any proper subterm $t''$ of $t'$.
Put $X=\Baire$.

\medskip

{\bf Case 1.}
The neck node of $\syn_t$ is labeled by a constant symbol $q$, i.e., $t$ is of the form $\phi_t(q)$.
Then, $\tpbf{\Sigma}_t=\tpbf{\Sigma}_q$.
By our assumption, $f$ is $\infty$-a.e.~$\lhd$-dominated by a $\tpbf{\Sigma}_t$ function.
However, such a function must be the constant function $x\mapsto q$.
Therefore, $\{x\in{\rm dom}(f):\neg (f(x)\lhd q)\}$ is $T_\infty$-meager.
Since ${\rm dom}(f)$ is $\Sigma^1_1$ and $f$ is $\Pi^1_1$-measurable w.r.t.~$\lhd$, this set is $\Sigma^1_1$, and in particular, $T_\infty$-open.
Therefore, it must be empty by the Baire category theorem for $T_\infty$ (Fact \ref{fact:GH-Baire}).
Hence, we have $f(x)\lhd q$ for any $x\in{\rm dom}(f)$.
This means that $f$ is $\lhd$-dominated by the constant function $g^\star\colon x\mapsto q$ on any $H\subseteq X$.
Clearly, $g^\star\in\Sigma_t(\Delta^1_1)$ in $H$.

\medskip

{\bf Case 2.}
The neck node of $\syn_t$ is labeled by $\mul$, i.e., $t$ is of the form $\phi_t(s\mul u)$.
By our assumption, $f$ is $\infty$-a.e.~$\lhd$-dominated by a total $\tpbf{\Sigma}_t$ function $g\colon H\to Z$.
Let $\bfS=(S_\sigma)_{\sigma\in\syn_t}$ be a flowchart on the term $t$ which determines $g$, i.e., $\eval{\bfS}=g$.
Let $\xi$ be the Borel rank of the neck node of $\syn_t$.
Note that $\xi<\om_1^{\rm CK}$ since $t$ is $\Delta^1_1$.
Then, a $\tpbf{\Sigma}^0_{\xi}$ set $U\subseteq X$ is assigned to the neck node $\sigma'$, i.e., $S_{\sigma'}=U$.
By Observation \ref{obs:Veblen-head-body}, note that we have $\eval{\bfS}\upto (H\setminus U)\in\tpbf{\Sigma}_{\phi_t(s)}$ and $\eval{\bfS}\upto (H\cap U)\in\tpbf{\Sigma}_{\phi_t(u)}$.
Then, define $U^+$ as the following $T_\xi$-open set:
\begin{align*}
%\notag
x\in U^+\iff&(\exists O\in\Sigma^1_1\cap\tpbf{\Sigma}^0_{\xi})(\exists g\in\tpbf{\Sigma}_{\phi_t(u)}(H\cap O))\\ 
&\qquad[x\in O\;\land\;(\forall^\infty y)\;(y\in {\rm dom}(f)\cap O\;\longrightarrow\;f(y)\lhd g(y))].
\end{align*}

It is not straightforward, but one can show that $U^+\subseteq X$ is the largest $T_{\xi}$-open set such that some $\tpbf{\Sigma}_{\phi_t(u)}$-function on $U^+\cap H$ $\infty$-a.e.~$\lhd$-dominates $f$ in the same way as in Claim \ref{sublem4} below.
First we show the following equivalences:
\begin{align*}
%\notag
x\in U^+\iff&(\exists O\in\Sigma^1_1\cap\tpbf{\Sigma}^0_{\xi})(\exists g^\star\in{\Sigma}_{\phi_t(u)}(\Delta^1_1;H\cap O))\\
&\qquad[x\in O\;\land\;(\forall y)\;(y\in {\rm dom}(f)\cap O\;\longrightarrow\;f(y)\lhd g^\star(y))]\\
%\notag
\iff&(\exists O^\star\in{\Sigma}^0_{\xi}(\Delta^1_1))(\exists g^\star\in{\Sigma}_{\phi_t(u)}(\Delta^1_1;H^\star\cap O))\\
&\qquad[x\in O\;\land\;(\forall y)\;(y\in {\rm dom}(f)\cap O^\star\;\longrightarrow\;f(y)\lhd g^\star(y))]
\end{align*}

The first equivalence follows from the induction hypothesis restricted to the $\Sigma^1_1$-domain $H\cap O$.
For the second equivalence, observe that a set $O$ in the first equivalent formula is always disjoint from $L=\{y\in{\rm dom}(f):\neg (f(y)\lhd g^\star(y))\}$.
By Lemma \ref{lem:complexity-flowchart}, $g^\star$ has a $\Delta^1_1$-measurable realizer $G\colon H\to\Baire$, i.e., $g^\star(y)=q_{G(y)}$.
As for the complexity of $L$, note that
\[y\in L\iff y\in{\rm dom}(f)\;\land\;(\exists z\in\Baire)\;[G(y)=z\;\land\;\neg(f(y)\lhd q_z)].\]

Therefore, the set $L$ is $\Sigma^1_1$, since ${\rm dom}(f)$ is $\Sigma^1_1$ and $f$ is $\Pi^1_1$-measurable w.r.t.~$\lhd$.
Hence, given a set $O$ in the first equivalent formula, since $O$ is $\Sigma^1_1$ and $\tpbf{\Sigma}^0_{\xi}$, by the Louveau separation theorem (Fact \ref{fact:Louveau-sep}), there exists a $\Sigma^0_{\xi}(\Delta^1_1)$ set $O^\star$ separating $O$ from $L$.
This verifies the third equivalence.
Next, we check that the second equivalent formula is $\Pi^1_1$.
Indeed:
\begin{claim}\label{sublem1}
$U^+$ can be obtained from a $\Pi^1_1$-sequence of $\Delta^1_1$-indices of $\tpbf{\Sigma}^0_{\xi}$ sets.
\end{claim}

\begin{proof}
To verify the claim, let $C$ be the set of codes of deterministic flowcharts $\bfS$ such that the domain of $\eval{\bfS}$ includes $H\cap O^\star$, and for $c\in C$, let $G_c\pcolon X \to\Baire$ be the realizer for the flowchart $\bfS_c$ coded by $c$, i.e., $\eval{\bfS_c}(x)=q_{G_c(x)}$ for any $x$ in the domain of $\eval{\bfS_c}$.
Then, the condition for $O^\star$ in the second equivalent formula for $x\in U^+$ can be rewritten as follows:
\[(\exists c\in\Delta^1_1)\;\left[c\in C\;\land\;(\forall y,z)\left(\left(y\in {\rm dom}(f)\cap O^\star\;\land\;G_c(y)=z\right)\;\longrightarrow\;f(y)\lhd q_z\right)\right]\]

To see that the formula inside the square brackets is $\Pi^1_1$, recall from Lemma \ref{lem:complexity-flowchart} that the set $C$ is $\Pi^1_1$ since $O^\star\cap H$ is $\Sigma^1_1$, and $G_c$ is $\Delta^1_1$-measurable relative to $c$.
Thus, since ${\rm dom}(f)$ is $\Sigma^1_1$ and $f$ is $\Pi^1_1$-measurable w.r.t.~$\lhd$, the inner formula is $\Pi^1_1$.
Hence, by the usual hyperarithmetical quantification argument (see \cite[Lemma III.3.1]{Sac90}), the whole formula is also $\Pi^1_1$.
This verifies the claim.
\end{proof}

Hereafter, we use $H'$ to denote the domain of $f$.

\begin{claim}\label{sublem2}
The function $f\upto H'\setminus U^+$ is $\lhd$-dominated by a ${\Sigma}_{\phi_t(s)}(\Delta^1_1)$-function on $H\setminus U^+$.
\end{claim}

\begin{proof}
By Fact \ref{approximate}, there exists a $T_\xi$-open set $V$ which is equal to $U$ $\infty$-a.e.
Since $V$ is $T_\xi$-open, if $x\in V$, then there exists $O\in\Sigma^1_1\cap\tpbf{\Sigma}^0_\xi$ such that $x\in O\subseteq V$.
Note that $O$ is $\infty$-a.e.~included in $U$, and therefore, $f\upto H'\cap O$ is $\infty$-a.e.~$\lhd$-dominated by $\eval{\bfS}\upto H\cap U\in\tpbf{\Sigma}_{\phi_t(u)}$.
Hence, by the definition of $U^+$, we have $O\subseteq U^+$, and therefore $x\in U^+$ for any $x\in V$.
Since $x$ is an arbitrary element of $V$, we have $U\subseteq U^+$ $\infty$-a.e.
In particular, $H\setminus U^+\subseteq H\setminus U$ $\infty$-a.e., and thus $f\upto H'\setminus U^+$ is $\infty$-a.e.~$\lhd$-dominated by $\eval{\bfS}\upto(H\setminus U)\in\mathbf{\Sigma}_{\phi_t(s)}$.
By Claim \ref{sublem1}, $U^+$ is $\Pi^1_1$, and therefore, by the induction hypothesis restricted to the $\Sigma^1_1$-domain $H\setminus U^+$, the function $f\upto H'\setminus U^+$ is now $\lhd$-dominated by some function in ${\Sigma}_{\phi_t(s)}(\Delta^1_1)$ on $H\setminus U^+$.
\end{proof}

\medskip

Let $\psi\pcolon\om^2\to\om$ be a partial $\Pi^1_1$ function parametrizing all $\Delta^1_1$ functions, i.e., for any $\Delta^1_1$-function $d\colon\om\to\om$, there exists $e\in\om$ such that $\psi(e,n)=d(n)$ for any $n\in\om$.
Let $I$ be the set of all indices $e\in\om$ such that the function $\psi_e$ defined by $\psi_e(n)=\psi(e,n)$ is total.
Clearly, $I$ is $\Pi^1_1$.
Given $P\subseteq I$, we use the notation $[P]_\xi$ to denote the union of $\tpbf{\Sigma}^0_\xi$ set whose codes have $\Delta^1_1$-indices in $P$, i.e., $[P]_\xi=\bigcup_{e\in P}\delta_\xi(\psi_e)$, where $\delta_\xi(p)$ is the $\tpbf{\Sigma}^0_\xi$ set coded by $p$.
By Claim \ref{sublem1}, there exists a $\Pi^1_1$ set $P^+\subseteq I$ such that $U^+=[P^+]_\xi$.
Let $\Phi(Y)$ be the formula saying that $f\upto H'\setminus[Y\cap I]_\xi$ is $\lhd$-dominated by a $\Sigma_{\phi_t(s)}(\Delta^1_1)$ function on $H\setminus[Y\cap I]_\xi$.

\begin{claim}\label{sublem3}
$\Phi$ is $\Pi^1_1$ on $\Pi^1_1$.
\end{claim}

\begin{proof}
Let $C(Y)$ be the set of codes of deterministic flowcharts $\bfS$ such that the domain of $\eval{\bfS}$ includes $H\setminus[Y\cap I]_\xi$.
If $(Y_n)_{n\in\om}$ is uniformly $\Pi^1_1$, then clearly $(H\setminus[Y_n\cap I]_\xi)_{n\in\om}$ is uniformly $\Sigma^1_1$.
Hence, by Lemma \ref{lem:complexity-flowchart}, $(C(Y_n))_{n\in\om}$ is uniformly $\Pi^1_1$.
Then, $\varphi(Y_n)$ holds if and only if
\begin{align*}
(\exists c\in\Delta^1_1)\;&\big[c\in C(Y_n)\;\land\;(\forall y,z)\\
&\left(\left(y\in {\rm dom}(f)\cap[Y_n\cap I]_\xi\;\land\;G_c(y)=z\right)\;\longrightarrow\;f(y)\lhd q_z\right)\big]
\end{align*}

By the same argument as in the previous claim, one can see that this is a $\Pi^1_1$ property uniformly in $n\in\om$.
Hence, $\Phi$ is $\Pi^1_1$ on $\Pi^1_1$.
\end{proof}

Recall from Claim \ref{sublem2} that $\Phi(P^+)$ holds since $U^+=[P^+]_\xi$.
Hence, by $\Pi^1_1$-reflection (Fact \ref{fact:reflection}), there exists a $\Delta^1_1$ set $R\subseteq P^+$ such that $\Phi(R)$ holds.
Then, $[R]_\xi$ is $\Sigma^0_{\xi}(\Delta^1_1)$.

\begin{claim}\label{sublem4}
The function $f\upto H'\cap [R]_\xi$ is $\lhd$-dominated by a $\Sigma_{\phi_t(u)}(\Delta^1_1)$-function on $H\cap[R]_\xi$.
\end{claim}

\begin{proof}
Fix a $\Delta^1_1$-enumeration of $R=(r_n)_{n\in\om}$, and put $O^\star_n=[\{r_n\}]_\xi$.
Note that $(O^\star_n)_{n\in\om}$ is a $\Delta^1_1$-sequence of $\tpbf{\Sigma}^0_\xi$ sets which covers $[R]_\xi$.
Since $X=\Baire$ is zero-dimensional, as seen in the proof of Proposition \ref{prop:zero-dim-partition}, one can effectively find a $\Delta^1_1$-sequence of pairwise disjoint $\tpbf{\Sigma}^0_\xi$ sets $(O'_n)_{n\in\om}$ such that $O'_n\subseteq O_n$ for each $n\in\om$ and $(O'_n)_{n\in\om}$ covers $[R]_\xi$.
Since $R\subseteq P^+$ and $O'_n\subseteq O^\star_n$, for any $n\in\om$, the function $f\upto H'\cap O'_n$ is $\lhd$-dominated by a $\Sigma_{\phi_t(u)}(\Delta^1_1)$-function $g'_n\colon H\cap O'_n\to Q$.

Now, consider the formula $E(n,c)$ (where $n\in\om$ and $c\in\om^\om$) stating that $c\in\Delta^1_1$ and $f\upto H'\cap O'_n$ is $\lhd$-dominated by the $\Sigma_{\phi_t(u)}(\Delta^1_1)$-function on $H\cap O_n'$ determined by the flowchart coded by $c$.
Since the formula saying that $c$ is $\Delta^1_1$ is $\Pi^1_1$, and the set $H\cap O_n'$ is $\Sigma^1_1$ uniformly in $n\in\om$, by Lemma \ref{lem:complexity-flowchart}, we observe that the formula $E(n,c)$ is $\Pi^1_1$.
By the property of $O_n'$, $E$ is a total relation, that is, for any $n\in\om$ there exists $c\in\om^\om$ such that $E(n,c)$.
Hence, by $\Delta^1_1$-selection (Fact \ref{fact:selection}), there exists a $\Delta^1_1$ function $f\colon\om\to\om^\om$ such that $E(n,f(n))$ for any $n\in\om$.

By the definition of $E$, this function $f$ yields a uniform $\Delta^1_1$-sequence of flowcharts $(\bfS_n)_{n\in\om}$ on $\phi_t(u)$ such that each $\bfS_n$ determines a $\Sigma_{\phi_t(u)}(\Delta^1_1)$-function $g_n'\colon H\cap O_n'\to Z$ which $\lhd$-dominates $f\upto H'\cap O_n'$.
We write $J$ for the syntax tree of $\phi_t(u)$, and then each flowchart $\bfS_n$ is of the form $(S^n_\sigma)_{\sigma\in J}$.
Then, for each $\sigma\in J$, if $\sigma$ is labeled by $\mul$, we define $S'_\sigma=\bigcup_{n\in\om}O_n'\cap S^n_\sigma$; and if $\sigma$ is labeled by $\ad$, we define $S'_{\sigma,i}=\bigcup_{n\in\om}O'_n\cap S^n_{\sigma,i}$ for each $i\in\om$.
Since such a node extends the neck node, its Borel rank is greater than or equal to $\xi$.
Hence, $S'_\sigma$ and $S'_{\sigma,i}$ are $\Sigma_{\rank(\sigma)}(\Delta^1_1)$ sets.
Thus, $\bfS'=(S'_\sigma)_{\sigma\in J}$ is a $\Delta^1_1$-flowchart on $\phi_t(u)$.
Since $(O'_n)_{n\in\om}$ is pairwise disjoint, for any $x\in [R]_\xi$ there exists a unique $n\in\om$ such that $x\in O'_n$.
Then, $\sigma$ is a true path for $x$ w.r.t.~$\bfS'$ if and only if $\sigma$ is a true path for $x$ w.r.t.~$\bfS_n$.
Since $\bfS_n$ determines a function, it is deterministic; hence this shows that $\bfS'$ is also deterministic.
Let $g'$ be the function determined by the flowchart $\bfS'$.
Then, $g'$ is a $\Sigma_{\phi_t(u)}(\Delta^1_1)$-function whose domain includes $H\cap[R_\xi]$, and we have $g'(x)=\eval{\bfS'}(x)=\eval{\bfS_n}(x)=g_n'(x)$ whenever $x\in H\cap O_n'$.
By our choice of $g_n'$, this implies that $f\upto H'\cap O_n'$ is $\lhd$-dominated by $g'\upto H\cap O_n'$ for any $n\in\om$.
Consequently, $f\upto H'\cap [R]_\xi$ is $\lhd$-dominated by $g'$.
\end{proof}

Since $\Phi(R)$ holds, there exists a $\Sigma_{\phi_t(s)}(\Delta^1_1)$-function $g''\colon H\setminus[R]_\xi\to Z$ which $\lhd$-dominates $f\upto H'\setminus[R]_\xi$.
Let $\bfS''$ be a $\Delta^1_1$-flowchart on $\phi_t(s)$ determining $g''$, and let $\bfS'$ be the $\Delta^1_1$-flowchart on $\phi_t(u)$ obtained by Claim \ref{sublem4} which determines a $\Sigma_{\phi_t(u)}(\Delta^1_1)$-function $g\colon H\cap[R]_\xi\to Z$ which $\lhd$-dominates $f\upto H'\cap[R]_\xi$.
Since $[R]_\xi$ is in $\Sigma^0_\xi(\Delta^1_1)$, the straightforward combination of the $\Delta^1_1$-flowcharts $\bfS''$ and $\bfS'$ yields a $\Delta^1_1$-flowchart $\bfS^\star$ on $t=\phi_t(s\mul u)$.
More precisely, recall that $\sigma'$ is the neck node of $\syn_t$, and its Borel rank is $\xi$.
Then, define $S^\star_{\sigma'}=[R]_\xi$.
Moreover, for each node $\sigma\in\syn_t$, if $\sigma$ is of the form $\sigma'\fr 0\fr \tau$, define $S^\star_\sigma=S''_{\sigma'\fr\tau}$; and if $\sigma$ is of the form $\sigma'\fr 1\fr \tau$, define $S^\star_\sigma=S'_{\sigma'\fr\tau}$.
It is easy to see that $\bfS^\star=(S^\star_\sigma)_{\sigma\in\syn_t}$ is a $\Delta^1_1$-flowchart determining a $\Sigma_{t}(\Delta^1_1)$-function $g^\star\colon H\to Z$ which $\lhd$-dominates $f$.
This completes the proof for Case 2.

%Let $U_R=\bigcup_\alpha O_\alpha$ where $O_\alpha$ is a $T_\xi$-basic open set for each $\alpha$ and let $U_R[\beta]=\bigcup_{\alpha<\beta} O_\alpha$. Since $\varphi(U_R)$ holds, $\varphi(U_R[\alpha])$ for some $\alpha$ by $\Pi^1_1$-reflection (the lightface version of Kechris \cite[\text{35.10}]{Kechris}). Now, $U_R[\alpha]$ is $\Sigma^0_{1+\xi}(\Delta^1_1)$. Thus, we can extend $A$ by some function $A^*\in \Sigma_{\langle L\rightarrow R\rangle^\xi}(\Delta^1_1)$.

\medskip

{\bf Case 3.}
The neck node of $\syn_t$ is labeled by $\ad$, i.e., $t$ is of the form $\phi_t(\ad_{n\in\om}s_n)$.
By our assumption, $f$ is $\infty$-a.e.~$\lhd$-dominated by a total $\tpbf{\Sigma}_t$ function $g\colon H\to Z$.
Let $\bfS=(S_\sigma)_{\sigma\in\syn_t}$ be a flowchart on the term $t$ which determines $g$, i.e., $\eval{\bfS}=g$.
Let $\xi$ be the Borel rank of the neck node of $\syn_t$.
Note that $\xi<\om_1^{\rm CK}$ since $t$ is $\Delta^1_1$.
Then, a $\tpbf{\Sigma}^0_{\xi}$ cover $S_n$ of $H$ is assigned to the neck node $\sigma'$, i.e., $S_{\sigma'}=(S_n)_{n\in\om}$.
By Observation \ref{obs:Veblen-head-body}, note that we have $\eval{\bfS}\upto (H\cap S_n)\in\tpbf{\Sigma}_{\phi_t(s_n)}$.
Then, define $U_n^+$ as the following $T_\xi$-open set:
\begin{align*}
%\notag
x\in U_n^+\iff&(\exists O\in\Sigma^1_1\cap\tpbf{\Sigma}^0_{\xi})(\exists g\in\tpbf{\Sigma}_{\phi_t(s_n)}(H\cap O))\\ 
&\qquad[x\in O\;\land\;(\forall^\infty y)\;(y\in {\rm dom}(f)\cap O\;\longrightarrow\;f(y)\lhd g(y))].
\end{align*}

As in the Case 2, intuitively, $U_n^+\subseteq X$ is the largest $T_{\xi}$-open set such that some $\tpbf{\Sigma}_{\phi_t(s_n)}$-function on $H\cap U_n^+$ $\infty$-a.e.~$\lhd$-dominates $f$.
Then, we have the following equivalences:
\begin{align*}
%\notag
x\in U_n^+\iff&(\exists O\in\Sigma^1_1\cap\tpbf{\Sigma}^0_{\xi})(\exists g^\star\in{\Sigma}_{\phi_t(s_n)}(\Delta^1_1;H\cap O))\\
&\qquad[x\in O\;\land\;(\forall y)\;(y\in {\rm dom}(f)\cap O\;\longrightarrow\;f(y)\lhd g^\star(y))]\\
%\notag
\iff&(\exists O^\star\in{\Sigma}^0_{\xi}(\Delta^1_1))(\exists g^\star\in{\Sigma}_{\phi_t(s_n)}(\Delta^1_1;H\cap O^\star))\\
&\qquad[x\in O\;\land\;(\forall y)\;(y\in {\rm dom}(f)\cap O^\star\;\longrightarrow\;f(y)\lhd g^\star(y))]
\end{align*}

The first equivalence follows from the induction hypothesis restricted to the $\Sigma^1_1$-domain $H\cap O$.
The second equivalence follows from Lemma \ref{lem:complexity-flowchart} and Louveau's separation theorem (Fact \ref{fact:Louveau-sep}) as in the Case 2.
As in Claim \ref{sublem1}, one can see that $U^+_n$ is obtained from a $\Pi^1_1$-sequence of $\Delta^1_1$-indices of $\tpbf{\Sigma}^0_{\xi}$ sets, uniformly in $n\in\om$.
In other words, there exists a $\Pi^1_1$ set $P^+\subseteq\om\times I$ such that $U_n^+=[P^+_n]_\xi$, where $P^+_n=\{e\in I:(n,e)\in P^+\}$, where $I$ is the set defined as in Case 2.

\begin{claim}\label{sublemlim1}
$(U^+_n)_{n\in\om}$ covers $H$.
\end{claim}

\begin{proof}
By Fact \ref{approximate} there is some $T_{\xi}$ open set which is equal to $S_n$ $\infty$-a.e. for each $n$.
Then, as in the proof of Claim \ref{sublem2}, one can see that $S_n\subseteq U_n^+$ $\infty$-a.e.
Since $(S_n)_{n\in\om}$ covers $H$, we have $H\setminus\bigcup_{n\in\om}U_n^+=\bigcup_{n\in\om}S_n\setminus\bigcup_{n\in\om}U_n^+\subseteq\bigcup_{n\in\om}(S_n\setminus U_n^+)$.
This set is $T_\infty$-meager since a countable union of $T_\infty$-meager sets is $T_\infty$-meager.
However, $H\setminus \bigcup_n U_n^+$ is $\Sigma^1_1$ and thus $T_\infty$-open. By the Baire category theorem for $T_\infty$ (Fact \ref{fact:GH-Baire}), this is in fact an empty set.
Consequently, $(U_n^+)_{n\in\om}$ covers $H$.
\end{proof}

For $Y\subseteq\om^2$, put $(Y)_n=\{e\in\om:(n,e)\in Y\}$ for each $n\in\om$.
Let $\Phi(Y)$ be the formula stating that $H\subseteq\bigcup_{n\in\om}[(Y\cap(\om\times D))_n]_\xi$.
For any uniform sequence of $\Pi^1_1$ sets $(P_k)_{k\in\om}$, it is easy to see that $\Phi(P_k)$ is $\Pi^1_1$ uniformly in $n\in\om$ since $H$ is $\Sigma^1_1$.
Hence, $\Phi$ is $\Pi^1_1$ on $\Pi^1_1$.
By Claim \ref{sublemlim1}, $\Phi(P^+)$ holds.
Therefore, by $\Pi^1_1$-reflection (Fact \ref{fact:reflection}), there exists a $\Delta^1_1$ set $R\subseteq P^+$ such that $\Phi(R)$ holds.
%Then, $([(R)_n]_\xi)_{n\in\om}$ is a $\Delta^1_1$-sequence of $\tpbf{\Sigma}^0_\xi$ sets which covers $H$.

Fix a $\Delta^1_1$-enumeration of $(R)_n=(r_{n,k})_{k\in\om}$, and put $O^\star_{n,k}=[\{r_{n,k}\}]_\xi$.
Note that $(O^\star_{n,k})_{n,k\in\om}$ is a $\Delta^1_1$-sequence of $\tpbf{\Sigma}^0_\xi$ sets which covers $H$.
Since $X=\Baire$ is zero-dimensional, as seen in the proof of Proposition \ref{prop:zero-dim-partition}, one can effectively find a $\Delta^1_1$-sequence of pairwise disjoint $\tpbf{\Sigma}^0_\xi$ sets $(O'_{n,k})_{n\in\om}$ such that $O'_{n,k}\subseteq O^\star_{n,k}$ for any $n,k\in\om$ and $(O'_{n,k})_{n\in\om}$ covers $H$.
Since $(R)_n\subseteq P_n^+$ and $O'_{n,k}\subseteq O^\star_{n,k}$, for any $n,k\in\om$, the function $f\upto H'\cap O'_{n,k}$ is $\lhd$-dominated by a $\Sigma_{\phi_t(s_n)}(\Delta^1_1)$-function $g'_{n,k}\colon H\cap O'_{n,k}\to Z$.
Here, $H'$ is the domain of $f$ as in Case 2.

Put $S^\star_n=\bigcup_{k\in\om}O'_{n,k}$ for each $n\in\om$.
Then, $S^\star_n$ is $\Sigma^0_\xi(\Delta^1_1)$, and as in Claim \ref{sublem4}, one can show that $f\upto H'\cap S^\star_n$ is $\lhd$-dominated by a $\Sigma_{\phi_t(s_n)}(\Delta^1_1)$-function on $H\cap S^\star_n$ for each $n\in\om$.
Again, as in the proof of Claim \ref{sublem4}, consider the formula $E(n,c)$ (where $n\in\om$ and $c\in\om^\om$) stating that $c\in\Delta^1_1$ and $f\upto H'\cap S^\star_n$ is $\lhd$-dominated by the $\Sigma_{\phi_t(s_n)}(\Delta^1_1)$-function on $H\cap S^\star_n$ determined by the flowchart coded by $c$.
As seen in the proof of Claim \ref{sublem4}, the formula $E(n,c)$ is $\Pi^1_1$, and by the property of $S^\star_n$, $E$ is a total relation.
Hence, by $\Delta^1_1$-selection (Fact \ref{fact:selection}), there exists a $\Delta^1_1$ function $f\colon\om\to\om^\om$ such that $E(n,f(n))$ for any $n\in\om$.

By the definition of $E$, this function $f$ yields a uniform $\Delta^1_1$-sequence of flowcharts $(\bfS_n)_{n\in\om}$ such that each $\bfS_n=(S^n_\sigma)_{\sigma}$ determines a $\Sigma_{\phi_t(s_n)}(\Delta^1_1)$-function $g_n'\colon H\cap O_n'\to Z$ which $\lhd$-dominates $f\upto H'\cap O_n'$.
Since $S^\star_n$ is in $\Sigma^0_\xi(\Delta^1_1)$ uniformly in $n\in\om$, the straightforward combination of the $\Delta^1_1$-flowcharts $(\bfS_n)_{n\in\om}$ yields a $\Delta^1_1$-flowchart $\bfS^\star$ on $t=\phi_t(\ad_{n\in\om}s_n)$.
More precisely, recall that $\sigma'$ is the neck node of $\syn_t$, and its Borel rank is $\xi$.
Then, define $S^\star_{\sigma',n}=S^\star_n$.
Moreover, for each node $\sigma\in\syn_t$, if $\sigma$ is of the form $\sigma'\fr n\fr \tau$, define $S^\star_\sigma=S^n_{\sigma'\fr\tau}$.
Since $(V_n)_{n\in\om}$ is a $\Delta^1_1$-sequence of pairwise disjoint $\tpbf{\Sigma}^0_\xi$ sets, as in Claim \ref{sublem4}, one can show that $\bfS^\star=(S^\star_\sigma)_{\sigma\in\syn_t}$ is a $\Delta^1_1$-flowchart determining a $\Sigma_{t}(\Delta^1_1)$-function $g^\star\colon H\to Z$ which $\lhd$-dominates $f$.
This completes the proof.
\end{proof}

\begin{proof}[Proof of Lemma \ref{thm:main-theorem-multi}]
%The argument is almost the same as in the proof of Lemma \ref{thm:main-theorem-multi}.
In the proof of Lemma \ref{thm:main-theorem}, the assumption $X=\om^\om$ (i.e., zero-dimensionality of the space) is only used to ensure that several flowcharts are deterministic.
For this reason, if $X$ is not necessarily zero-dimensional, then we cannot ensure that the resulting function $g^\star$ is single-valued.
To verify the assertion, we need to modify the proof of Lemma \ref{thm:main-theorem} to conform to multi-valued functions; for instance, we need to replace the condition $G_c(y)=z$ in Claim \ref{sublem1} with $z\in G_c(y)$.
However, as in the proof of Lemma \ref{lem:complexity-flowchart}, one can see that $q_z\in h(x)$ is a $\Delta^1_1$ condition whenever $h$ is in $\Sigma_s(\Delta^1_1)$ where $s$ is an $\lang(Q)$-term.
Thus, in any complexity calculation (e.g.~Claims \ref{sublem1}, \ref{sublem3} and \ref{sublem4}), involving multi-valued functions does not increase the complexity.
Therefore, we can make exactly the same argument as in the proof of Lemma \ref{thm:main-theorem}.
\end{proof}

%\begin{question}
%Does Theorem \ref{thm:Louveau-for-functions} hold for an arbitrary computable quasi-Polish space?
%\end{question}

\begin{question}
Does Theorems \ref{thm:main-theorem-extension} and \ref{thm:main-theorem-domination} hold for an arbitrary computable Polish space instead of $\Baire$?
\end{question}

\bibliographystyle{plain}
\bibliography{Louveau}
\end{document}